\documentclass{amsart}

\usepackage{amsthm, amssymb, amsmath, amscd}
\usepackage{eurosym}
\usepackage{amsfonts}
\usepackage{amsmath}
\usepackage{bbm}
\usepackage{setspace}
\usepackage{graphicx}
\usepackage{tensor}
\usepackage{color}

\usepackage[all,cmtip]{xy}

\usepackage{indentfirst}

\newtheorem{theorem}{Theorem}[section]
\newtheorem{thm*}{Theorem}

\newtheorem{lemma}[theorem]{Lemma}

\newtheorem{prop}[theorem]{Proposition}

\theoremstyle{definition}
\newtheorem{define}[theorem]{Definition}
\newtheorem{ex}[theorem]{Example}
\newtheorem{remark}[theorem]{Remark}

\newcommand{\field}[1]{\mathbb{#1}}

\newcommand{\zkz}{\field{Z}/k\field{Z}}

\newcommand{\Z}{\field{Z}}
\newcommand{\C}{\field{C}}

\begin{document}
\title[Relative geometric assembly and mapping cones, Part I]{Relative geometric assembly and mapping cones,\\
Part I: The geometric model and applications}
\author{Robin J. Deeley, Magnus Goffeng}
\date{\today}

\begin{abstract}
Inspired by an analytic construction of Chang, Weinberger and Yu, we define an assembly map in relative geometric $K$-homology. The properties of the geometric assembly map are studied using a variety of index theoretic tools (e.g., the localized index and higher Atiyah-Patodi-Singer index theory). As an application we obtain a vanishing result in the context of manifolds with boundary and positive scalar curvature; this result is also inspired and connected to work of Chang, Weinberger and Yu. Furthermore, we use results of Wahl to show that rational injectivity of the relative assembly map implies homotopy invariance of the relative higher signatures of a manifold with boundary. 
\end{abstract}

\maketitle

\section{Introduction}

We construct and study an assembly map in relative geometric (i.e., Baum-Douglas) $K$-homology. Our results are inspired and related to the analytic relative assembly map constructed via localization algebras by Chang, Weinberger, and Yu in \cite[Section 2]{CWY}. However, in the introduction, we concentrate on the geometric aspects of our construction; localization algebras and the construction in \cite[Section 2]{CWY} are reviewed in detail in Section \ref{CWYmap}. One of the advantages of the geometric model of $K$-homology is the explicit construction of Chern characters, see for example \cite[Part 5]{BD}. Building on the results of the current paper, we construct Chern characters in the relative case in \cite{DGrelII}. The construction in \cite{DGrelII} is similar to the one in \cite[Section 4]{DGIII}. The assembly maps we consider in this paper are for free actions and coincides with Baum-Connes' assembly mapping for proper actions \cite{bch} when the involved groups are torsion-free.

We begin with a brief introduction to geometric $K$-homology and then discuss the five main results of the paper; three of these results are explicitly stated. Geometric $K$-homology (see any of \cite{BCW,BD,BE, Rav, Wal}) provides an alternative realization of certain $KK$-groups. If $X$ is a finite $CW$-complex and $A$ is a unital $C^*$-algebra, then
\[
KK^*(C(X), A) \cong K^{\rm geo}_*(X; A) := \left\{ (M, E_A, f)\right\}/\sim
\]
where $M$ is a smooth compact spin$^c$-manifold (without boundary), $E_A\to M$ is a smooth $A$-bundle (i.e., $E_A$ is a locally trivial bundle with fibers given by finitely generated, projective Hilbert $A$-modules), $f:M \rightarrow X$ is a continuous map, and $\sim$ is the geometrically defined Baum-Douglas relation (i.e., the equivalence relation generated from bordism, vector bundle modification and direct sum/disjoint union). When $X$ is a $CW$-complex which is only locally finite, $K_*^{\rm geo}(X;A)$ is isomorphic to the compactly supported $KK$-theory group.

Let $\Gamma$ be a finitely generated discrete group and $B\Gamma$ its classifying space (which we assume is a locally finite $CW$-complex). The maximal group $C^*$-algebra of $\Gamma$ is denoted by $C^*_{\bf max}(\Gamma)$. Then, we can form $K^{\rm geo}_*(B\Gamma; \C)$ (which we denote by $K^{\rm geo}_*(B\Gamma)$) and $K^{\rm geo}_*(pt; C^*_{\bf max}(\Gamma))$. Furthermore, the maximal Baum-Connes assembly map for free actions is defined at the level of these geometric cycles via
\begin{equation}
\label{geometricassemblyequation}
\mu_{\textnormal{geo}} : (M,E_{\field{C}}, f) \mapsto (M, E_{\field{C}}\otimes_{\field{C}} f^*(\mathcal{L}_{B\Gamma})).
\end{equation}
By the definition of a cycle for $K_*^{\rm geo}(B\Gamma)$ (see above), $f:M\to B\Gamma$ is a continuous mapping. Here $\mathcal{L}_{B\Gamma}\to B\Gamma$ denotes the Mischenko bundle: 
\begin{equation}
\label{mishdef}
\mathcal{L}_{B\Gamma}:= E\Gamma \times_{\Gamma} C^*_{\bf max}(\Gamma)\to B\Gamma.
\end{equation}
The reader can find more on this map in for example \cite[Introduction]{DG} and references therein. Under the isomorphisms with the analytic models, the geometrically defined assembly mapping \eqref{geometricassemblyequation} coincides with the standard construction of assembly (see \cite{land}).

Our goal is to extend this geometric construction to the relative case. That is, to define a geometric assembly map given as input a group homomorphism $\phi: \Gamma_1 \rightarrow \Gamma_2$ rather than a single group $\Gamma$. The homomorphism $\phi$ induces maps
\[
B\phi:B\Gamma_1\to B\Gamma_2 \hbox{ and }\phi: C^*_{\bf max}(\Gamma_1) \to C^*_{\bf max}(\Gamma_2).
\]
Note the abuse of notation used in the definition of the latter.

The domain and codomain of our assembly map are introduced in detail in Section \ref{GeoMap} (in particular, see subsections \ref{domainsubs} and \ref{codomainsubs}). They are denoted respectively by $K_*^{\textnormal{geo}}(B\phi) \hbox{ and }K_*^{\textnormal{geo}}(pt; \phi)$. These abelian groups are defined using geometric cycles and the relative geometric assembly map, $\mu_{\textnormal{geo}}$, is defined explicitly at the level of cycles, see Definition \ref{defofassembly}. The compatibility between the relative groups/relative assembly map and absolute groups/absolute assembly map within the geometric context is made precise in Theorem \ref{sesforbphi}. The reader who is only interested in the geometric definition of the relative assembly need only read Section \ref{GeoMap}, which is self-contained.

It is natural to compare the geometrically defined assembly map to the analytic assembly map of Chang, Weinberger and Yu defined using asymptotic morphisms (see \cite{connesbook,conneshigson}). In particular, Theorem \ref{sesforbphi} should be compared with \cite[Theorem 2.17]{CWY}. To make such comparisons precise, one must first construct explicit isomorphisms at the level of the $K$-homology groups defining the domain and codomain of these maps.  For the domains, this is basically the standard isomorphism between analytic and geometric $K$-homology. We review its construction, which uses the localized index (see subsection \ref{localindex}) in Theorem \ref{relaassonkhom}.

The isomorphism between the codomains is more involved. It fits into the general context of any unital $*$-homomorphism. As such, for the moment, suppose $\phi: B_1 \rightarrow B_2$ is a unital $*$-homomorphism between unital $C^*$-algebras; of course, the case of $\phi: C^*_{\bf max}(\Gamma_1) \rightarrow C^*_{\bf max}(\Gamma_2)$ is the most relevant. One can form the mapping cone $C^*$-algebra associated to $\phi$:
\[
C_{\phi}:=\{ (f, a) \in   C_0((0,1],B_2)\oplus B_1\: | \: f(1)= \phi(a) \}
\]

\begin{thm*} 
\label{firstthm}
The construction of $\Phi_{{\rm cone}}$ in Equation \eqref{definofphicone} on page \pageref{definofphicone} produces a well defined isomorphism $\Phi_{{\rm cone}}:K_*^{\textnormal{geo}}(pt;\phi)\to K_{*+1}(C_\phi)$ fitting into a commutative diagram with exact rows and vertical mappings being isomorphisms:
\begin{equation}
 \minCDarrowwidth15pt
\begin{CD}
@>\phi_*>> K_*^{\textnormal{geo}}(pt;B_2) @>r>> K_*^{\textnormal{geo}}(pt;\phi) @>\delta >> K_{*+1}^{\textnormal{geo}}(pt;B_1) @>\phi_*>> \\
@.  @VVV  @VV\Phi_{\textnormal{cone}} V  @VV V \\
@>\phi_*>> K_*(B_2) @>r>> K_{*+1}(C_\phi) @>\delta >> K_{*+1}(B_1) @>\phi_*>> \\
\end{CD}
\end{equation}
where the top row is from \cite[Theorem 3.13]{DeeRZ} and the bottom row is the long exact sequence in $K$-theory obtained from the short exact sequence of $C^*$-algebras:
\[ 0 \rightarrow SB_2 \rightarrow C_{\phi} \rightarrow B_1 \rightarrow 0
\]
\end{thm*}
This result is a substantial improvement of results in \cite{DeeMappingCone}: firstly, the assumption that $\phi_*: K_*(B_1) \rightarrow K_*(B_2)$ is injective has been removed and secondly, in the definition $\Phi_{\rm cone}$, one can choose any trivializing operator to construct the higher APS-index. Higher APS-index theory is reviewed before the construction of $\Phi_{\rm cone}$, see Section \ref{IsoGeoToAna}. The definition of a trivializing operator can be found in Definition \ref{diracsandtrivi} (see page \pageref{diracsandtrivi}).

\subsubsection*{Vanishing results for PSC-metrics} Returning to the special case of $\phi : C^*_{\bf max}(\Gamma_1) \rightarrow C^*_{\bf max}(\Gamma_2)$, obtained from a group homomorphism, we state an application of these theorems (i.e., the previous theorem and Theorem \ref{sesforbphi}). It is stated as Theorem \ref{apswhenpsconbdry} in the body of the paper. An essentially equivalent result was proven as Theorem 2.18 in \cite{CWY} using localization algebras. Our proof uses higher APS-theory and conceptually explains this relative vanishing result through the two steps in the proof: firstly, the existence of a metric of positive scalar curvature near the boundary of a manifold implies that the relative assembly of said manifold is realized as a higher APS-index (Proposition \ref{apswhenpsconbdry}) and secondly that a metric of positive scalar curvature in the interior guarantees the vanishing of the higher APS-index (Lemma \ref{vanishingofaps}). 

\begin{thm*}
\label{thmpscone}
Let $W$ be a connected spin-manifold with boundary and let $[W]\in K_*(W,\partial W)$ denote the fundamental class. If $W$ admits a metric of positive scalar curvature which is collared at the boundary, then 
$$\mu^\phi_{\rm geo}[W]=0 \in K_{{\rm dim}(W)}(pt;\phi)$$
where $\phi: \pi_1(\partial W) \rightarrow \pi_1(W)$ is the group homomorphism induced from the inclusion $i:\partial W\hookrightarrow W$.
\end{thm*}

\subsubsection*{The strong Novikov property} One of the main applications of the assembly map is to the Novikov conjecture on homotopy invariance of relative higher signatures. In the relative setting, this problem has been discussed in \cite[Section 12]{weinnov} and \cite[Section 4.9]{Lott}. The Novikov conjecture follows from rational injectivity of the free assembly map. With this implication as motivation, we say that a group homomorphism $\phi$ has the strong relative Novikov property if $\mu_{\rm geo}^\phi:K_*^{\rm geo}(B\phi)\to K_*(pt;\phi)$ is rationally injective. The strong Novikov conjecture for $\Gamma_1$ and $\Gamma_2$ has no obvious implication for the strong relative Novikov property of a homomorphism $\phi:\Gamma_1\to \Gamma_2$. Neither does the Baum-Connes conjecture for $\Gamma_1$ and $\Gamma_2$ if there is torsion present. The relation to homotopy invariance of relative higher signatures is given in the following theorem appearing as Theorem \ref{homotopyandnovi} below. Its proof relies on a result of Wahl \cite{WahlSurSet}.

\begin{thm*}
\label{introhomotopyandnovi}
Let $\phi:\Gamma_1\to \Gamma_2$ be a group homomorphism with the strong relative Novikov property. For any manifold with boundary $W$ and mappings $f:W\to B\Gamma_2$, $g:\partial W\to B\Gamma_1$ such that $f|_{\partial W}=B\phi\circ g$, the relative higher signatures
\begin{equation}
\mathrm{sign}_\nu(W ,(f,g)):=\int_W (f,g)^*(\nu)\wedge L(W), \quad\nu\in H^*(B\phi),
\end{equation}
are orientation preserving homotopy invariants of $[i:\partial W\to \overline{W}]$ (see Definition \ref{homotopiesofomor}). 
\end{thm*}

In the formulation of Theorem \ref{introhomotopyandnovi} we use the relative cohomology group $H^*(B\phi)$. For a short introduction to relative cohomology groups see Subsection \ref{subsecononen} on page \pageref{relcohompairfunc}. A key feature is that for $\nu\in H^*(B\phi)$, the class $(f,g)^*\nu$ can be represented by a compactly supported de Rham cohomology class on the interior $W^\circ$. The relative higher signatures contain no boundary contribution, in contrast to Atiyah-Patodi-Singer's formula for the signature of a manifold with boundary (see \cite[Theorem 4.14]{APSI}). The conceptual reason for this difference is that the compatibility condition $f|_{\partial W}=B\phi\circ g$ ensures that the boundary contributions are cancelled in the relative pairing.

\subsubsection*{Compatibility with the analytic map} With isomorphisms defined at the level of the groups as in Theorem \ref{firstthm}, we would like to compare the assembly maps. In particular, even without the full details of the definition of relative assembly using localization algebras, the reader would likely agree that asking whether the following diagram commutes is a natural question:
\begin{equation}
\label{commdiaformu}
\begin{CD}
K^{\textnormal{geo}}_*(B \phi) @>\mu_{\textnormal{geo}}^\phi >> K^{\textnormal{geo}}_*(pt;\phi)  \\
@V\mathrm{ind}_L^{\rm rel} VV @VV\Phi_{{\rm cone}}V   \\
K_{*+1}(C_{B \phi}) @>\mu_{CWY} >> K_{*+1}(C_{\psi}) \\
\end{CD}
\end{equation}
The horizontal maps are the assembly maps (defined respectively in Sections \ref{CWYmap} and \ref{GeoMap}) and the vertical maps are the isomorphisms defined respectively in Theorem \ref{relaassonkhom} and Equation \eqref{definofphicone} in Lemma \ref{cyctoclasses}.

In a previous version of this paper, we claimed the diagram \eqref{commdiaformu} commutes in general. However, there was a gap in the proof. We still believe this should hold in general, but since this question is not central to the applications we have in mind we only prove commutativity with additional hypotheses, see Theorem \ref{comDiaGeoVsCWY}. It remains an interesting open question whether these additional hypotheses can be removed. 

\subsection{Notation}
\label{notationseciont}
We use the following notation, some of which has already been introduced. The reader may wish to skip this section and return to it as needed.

We will consider a continuous map $h:Y\to X$. Quite often this map will be the inclusion $Y \subseteq X$ associated with a finite $CW$-pair. We use $\Gamma_1$ and $\Gamma_2$ to denote finitely generated discrete groups and $\phi:\Gamma_1 \rightarrow \Gamma_2$ denotes a group homomorphism between these groups. The classifying space of $\Gamma_i$ is denoted by $B\Gamma_i$. We also let $\phi$ denote the $*$-homomorphism induced from the group homomorphism $\phi$; it is a $*$-homomorphism from the maximal group $C^*$-algebra $C^*(\Gamma_1)$ of $\Gamma_1$ to the maximal group $C^*$-algebra $C^*(\Gamma_2)$ of $\Gamma_2$. The mapping cone of $\phi$ is the $C^*$-algebra: 
$$C_{\phi}:=\{ (f, a) \in   C_0((0,1],C^*(\Gamma_2))\oplus C^*(\Gamma_1)\: | \: f(1)= \phi(a) \},$$ 
where $C^*(\Gamma)$ denotes the maximal $C^*$-algebra of a group $\Gamma$. We note that this convention for the mapping cone is different than the one used in \cite{DeeRZ, DeeMappingCone} where the interval $[0,1)$ is used and also that this construction can be applied to any unital $*$-homomorphism. If $A$ is $C^*$-algebra, then the suspension of A is the $C^*$-algebra $SA:=C_0(\field{R},A)$. A representation of $A$ is called standard if it is faithful and no nonzero operator is represented by a compact operator. If $A$ is unital, then an $A$-bundle refers to a locally trivial bundle with fibers given by finitely generated, projective Hilbert $A$-modules. An example of such a bundle is the Mischchenko bundle, $\mathcal{L}_{B\Gamma}= E\Gamma \times_{\Gamma} C^*(\Gamma)$ from Equation \eqref{mishdef}. 

There is a continuous map $B\phi: B\Gamma_1 \rightarrow B\Gamma_2$ implementing $\phi$ at the level of homotopy groups. The map $B\phi$ is only unique up to homotopy, we fix one choice of $B\phi$. Furthermore, we assume that $B\Gamma_1$ and $B\Gamma_2$ are locally finite $CW$-complexes. The universal property for $E\Gamma_2$ produces a $\Gamma_2$-equivariant continuous map $\Gamma_2\times_{\Gamma_1}E\Gamma_1\to E\Gamma_2$ and by composing with the $\Gamma_1$-equivariant map $E\Gamma_1\to \Gamma_2\times_{\Gamma_1}E\Gamma_1$, $x\mapsto [1_{\Gamma_2},x]$, we can assume that $B\phi$ lifts to a $\Gamma_1$-equivariant mapping $E\phi:E\Gamma_1\to E\Gamma_2$. A morphism between two continuous mappings $h:Y\to X$ and $h':Y'\to X'$ is a pair $(f,g):[h:Y\to X]\to [h':Y'\to X']$ of mappings $f:X\to X'$ and $g:Y\to Y'$ such that $f\circ h=h'\circ g$.

The theory of localization algebras and assembly requires some knowledge of asymptotic morphisms, see \cite{connesbook, conneshigson,Yu} for details. We will also use the following notation when considering a manifold with boundary: if $W$ is a Riemannian manifold with boundary, the inclusion is denoted by $i:\partial W\to W$ and the interior by $W^\circ$. When we wish to emphasize the presence of a boundary, we write $\overline{W}=W^\circ\cup\partial W$. Furthermore, we assume all geometric structures are of product type near the boundary and identify a collar neighbourhood of the boundary with $(0,1]\times \partial W$. The boundary $\partial W$ will be identified with $\{1\}\times \partial W\subseteq \overline{W}$. Let $\overline{W}_R:=\overline{W}\cup_{\partial W}[1,R]\times \partial W$, so $\overline{W}_1=\overline{W}$. We also write $W_\infty:=\overline{W}\cup_{\partial W}[1,\infty)\times \partial W$. We equipp $\overline{W}_R$ and $W_\infty$ with the cylindrical metric on the attached cylinder. If $W$ has a spin$^c$-structure, we let $S_W\to W$ denote the associated Clifford bundle of complex spinors.

\section{The relative model of Chang, Weinberger, and Yu}
\label{CWYmap}

We begin by recalling some notation and results from the literature on localization algebras, assembly and localized indices. The results in this section can be found in, or readily deduced from, the work of Yu with coauthors \cite{CWY,gongwangyu,willettyu12,Yu} and Qiao-Roe \cite{qr10}. The careful reader notes that the methods of \cite{qr10} apply to the reduced $C^*$-norm, but the recent work \cite{SchickSeyedhosseini} shows that the same methods apply when using the maximal norm.

\subsection{Localization algebras and assembly} 
\label{localsubs}
Yu defines the assembly map using localization algebras. We give a quick review of this theory; the reader is directed to \cite{Yu} for the full details. Let $(Z,\mathrm{d}_Z)$ be a second countable proper metric space. For purposes that will be clearer on the next page, we always assume that $(Z,\mathrm{d}_Z)$ has \emph{bounded geometry}
\footnote{Following \cite{gongwangyu,CWY}, we say that $Z$ has bounded geometry if there is a countable subset $Z_0\subseteq Z$ such that for some $c>0$, $\mathrm{d}_Z(x,Z_0)\leq c$ for all $x\in Z$ and for any $r>0$ there is an $N>0$ such that $\#\{x\in Z_0: \mathrm{d}_Z(x,x_0)<r\}\leq N$ for all $x_0\in Z_0$.}. 
We fix a representation, $\rho$, of $C_0(Z)$ as operators on a Hilbert space $\mathcal{H}_Z$. We tacitly assume that $\rho$ is a standard representation, that is, $\rho$ is faithful and $\rho(C_0(Z))\cap \mathbbm{K}(\mathcal{H}_Z)=0$.
\begin{itemize}
\item An operator $T$ on $\mathcal{H}_Z$ is said to be locally compact if $\rho(a)T$ and $T\rho(a)$ are compact for any $a\in C_0(Z)$.
\item If $T$ is an operator on $\mathcal{H}_Z$ and there exists an $R>0$ such that $aTb=0$ whenever $\mathrm{d}_Z(\mathrm{supp}(a),\mathrm{supp}(b))\geq R$, we say that $T$ has finite propagation. We let $\mathrm{prop}(T)$ denote the infimum of the set of such $R$; we call  $\mathrm{prop}(T)$ the propagation of $T$.
\item The Roe algebra, $\mathbb{R}_{\rho}(Z)\subseteq \mathbbm{B}(\mathcal{H}_Z)$, is the $*$-algebra of operators that are locally compact and have finite propagation. 
\end{itemize}
There are several choices of $C^*$-norms on $\mathbb{R}_\rho(Z)$ and different choices can lead to different $C^*$-algebras. Unless otherwise stated, we assume that $Z$ has bounded geometry and use the maximal completion of $\mathbb{R}_\rho(Z)$ (see \cite[Section 3]{gongwangyu}), we denote this by $C^*(Z)$. When $Z$ has bounded geometry, the maximal $C^*$-norm is well defined on $\mathbb{R}_\rho(Z)$ by \cite[Lemma 3.4, Item 4.4]{gongwangyu}. The assumption that $\rho$ is a standard representation guarantees that the $K$-theory of $C^*(Z)$ is independent of $\rho$ (see \cite[Section 4]{higroeyu}). Furthermore, the localization algebra(s) of Yu, $C^*_L(Z)$, is a suitable completion of the $*$-algebra $\mathbb{R}_{L,\rho}(Z)$ of the functions of the following form: $g: [1, \infty) \rightarrow \mathbb{R}_{\rho}(Z)$ such that
\begin{enumerate}
\item $g$ is uniformly bounded and uniformly continuous;
\item $\mathrm{prop}(g(s))\to 0$ as $s$ tends to infinity.
\end{enumerate}
Again, the $C^*$-norm used for the completion and in Item (1) above is important (i.e., lead to different algebras); we use the maximal $C^*$-norm unless otherwise stated. To be precise, the norm on $C^*_L(Z)$ is that induced from the $C^*$-algebra $C_{\textnormal{un}}([1,\infty),C^*(Z))$ of uniformly continuous functions $[1,\infty)\to C^*(Z)$. The choice of representation $\rho$ is less important as it does not affect the isomorphism class of relevant $K$-theory groups when we assume $\rho$ to be a standard representation, see \cite[Section 3]{Yu}. Our convention differs from that in \cite{CWY,xieyu} where the interval $[0,\infty)$ is used to define the localization algebras. 

Let $X$ be a finite $CW$-complex and denote its universal cover by $\tilde{X}$. We choose a word metric on $\pi_1(X)$. We assume that $X$ is equipped with a metric defining its topology and that $\tilde{X}$ is equipped with a metric making the $\pi_1(X)$-action isometric and each orbit bi-Lipschitz equivalent to $\pi_1(X)$. The action of the deck transformations $\pi_1(X)$ on $\tilde{X}$ defines an action of $\pi_1(X)$ as $*$-automorphisms on $\mathbb{R}_\rho(\tilde{X})$ and $\mathbb{R}_{L,\rho}(\tilde{X})$. The construction discussed in the previous paragraph can be applied in a number of ways: we can form the algebras $C^*_L(X)$, $C^*(\tilde{X})$, $C^*(\tilde{X})^{\pi_1(X)}$, and $C^*_L(\tilde{X})^{\pi_1(X)}$. The last two $C^*$-algebras are obtained using the $*$-algebra of $\pi_1(X)$-invariant elements in $\mathbb{R}_\rho(\tilde{X})$ and $\mathbb{R}_{L,\rho}(\tilde{X})$, respectively. That is, $C^*(\tilde{X})^{\pi_1(X)}$ is defined as the maximal $C^*$-completion of $\mathbb{R}_\rho(\tilde{X})^{\pi_1(X)}$ and $C^*_L(\tilde{X})^{\pi_1(X)}$ as the completion of $\mathbb{R}_{L,\rho}(\tilde{X})^{\pi_1(X)}$ in  $C_{\textnormal{un}}([1,\infty),C^*(\tilde{X})^{\pi_1(X)})$. 
The maximal $C^*$-norm on $\mathbb{R}_\rho(\tilde{X})^{\pi_1(X)}$ is well defined by \cite[Lemma 4.13]{gongwangyu} since there is a $\Gamma$-invariant discrete $\tilde{X}_0\subseteq \tilde{X}$ of bounded geometry such that for some $c>0$, $\mathrm{d}(x,\tilde{X}_0)\leq c$ for all $x\in \tilde{X}$.

Since we are working in the relative context, maps between localization algebras must also be considered. If $h: Y \rightarrow X$ is a coarse Lipschitz map, there is an induced map at the localization algebra level by \cite[Lemma 3.4]{Yu}. By an abuse of notation, we denote also said $*$-homomorphism by $h:C^*_L(Y)\to C^*_L(X)$.

\begin{remark}
As is observed in \cite[Section 3.2]{zeid14}, the functoriality holds in larger generality. If $\Gamma$ is a discrete group acting on $X$ and $Y$, making $h:Y\to X$ equivariant, there is an associated $*$-homomorphism $h:C^*_L(Y)^\Gamma\to C^*_L(X)^\Gamma$ of equivariant localization algebras if $h$ is a uniformly continuous coarse mapping. 
\end{remark}

The Lipschitz case suffices for our purposes: if $\mathrm{d}_X$ and $\mathrm{d}_Y$ are metrics defining the topology on $X$ and $Y$, respectively, $h:(Y,\mathrm{d}_Y+h^*\mathrm{d}_X)\to (X,\mathrm{d}_X)$ is a Lipschitz mapping and $\mathrm{d}_Y+h^*\mathrm{d}_X$ defines the same topology as $\mathrm{d}_Y$ because $h$ is continuous. We consider the mapping cone associated with $h:C^*_L(Y)\to C^*_L(X)$; it is denoted by $C_h$. If $Y$ is a compact space and $h:Y\to X$ is continuous we use the notation 
\begin{equation}
\label{cwynotation}
K_*^{\textnormal{CWY}}(Y):=K_*(C^*_L(Y))\quad\mbox{and}\quad K_*^{\textnormal{CWY}}(h):=K_*(SC_h).
\end{equation}
(recall that $SA:=C_0(\field{R},A)$ for a $C^*$-algebra $A$).

\begin{prop}
\label{findepocwy}
Let $h:Y\to X$ be a continuous mapping of compact spaces. The group $K_*^{CWY}(h)$ is up to canonical isomorphisms independent of the choice of metrics on $X$ and $Y$ and the choice of standard representations $\rho_X$ and $\rho_Y$ of $C_0(X)$ and $C_0(Y)$, respectively. 
\end{prop}

\begin{proof}
For $j=1,2$, suppose that $\mathrm{d}_{X,j}$ and $\mathrm{d}_{Y,j}$ are metrics on $X$ and $Y$, respectively, such that $h:(Y,\mathrm{d}_{Y,j})\to (X,\mathrm{d}_{X,j})$ is Lipschitz continuous. Define the metrics $\mathrm{d}_{X,3}:=\mathrm{d}_{X,1}+\mathrm{d}_{X,2}$ and $\mathrm{d}_{Y,3}:=\mathrm{d}_{Y,1}+\mathrm{d}_{Y,2}$. The map $h:(Y,\mathrm{d}_{Y,3})\to (X,\mathrm{d}_{X,3})$ and the inclusion maps $(Y,\mathrm{d}_{Y,3})\to (Y,\mathrm{d}_{Y,j})$ and $ (X,\mathrm{d}_{X,3})\to  (X,\mathrm{d}_{X,j})$, $j=1,2$, are by construction Lipschitz continuous. 

For fixed representations $\rho_X$ and $\rho_Y$ of $C_0(X)$ and $C_0(Y)$, respectively, we have inclusions $\mathbb{R}_{L,\rho_X}(X,\mathrm{d}_{X,3})\subseteq \mathbb{R}_{L,\rho_X}(X,\mathrm{d}_{X,j})$ and $\mathbb{R}_{L,\rho_Y}(Y,\mathrm{d}_{Y,3})\subseteq \mathbb{R}_{L,\rho_Y}(Y,\mathrm{d}_{Y,j})$ for $j=1,2$. Here we include the metric in the notation for the localization algebra for notational clarity. We therefore arrive at a commuting diagram (for $j=1,2$)
$$\begin{CD}
C^*_{L,\rho_Y}(Y,\mathrm{d}_{Y,3})@>h >> C^*_{L,\rho_X}(X,\mathrm{d}_{X,3})   \\
@VVV @VVV   \\
C^*_{L,\rho_Y}(Y,\mathrm{d}_{Y,j}) @>h >> C^*_{L,\rho_X}(X,\mathrm{d}_{X,j}).  \\
\end{CD}$$
We abuse the notation and write $h$ for both maps on the two different localization algebras. The vertical arrows are induced by the inclusion maps mentioned above and in turn induce isomorphisms on $K$-theory by \cite[Page 313]{Yu} (see also \cite[Proposition 3.2]{qr10}). We therefore arrive at a commuting diagram with exact rows for $j=1,2$:
$$\begin{CD}
0@>>> SC^*_{L,\rho_X}(X,\mathrm{d}_{X,3}) @>>> C_h @>>> C^*_{L,\rho_Y}(Y,\mathrm{d}_{Y,3}) @>>>0 \\
@.  @VVV  @VVV  @VV V \\
0@>>> SC^*_{L,\rho_X}(X,\mathrm{d}_{X,j}) @>>> C_h @>>> C^*_{L,\rho_Y}(Y,\mathrm{d}_{Y,j}) @>>>0 \\\\
\end{CD}$$
Since the outer most vertical arrows induce isomorphisms on $K$-theory, the same holds for the middle vertical arrow and we deduce that $K_*(C_h)$ (up to canonical isomorphism) does not depend on the choice of metrics on $X$ and $Y$. A similar argument, again using \cite[Page 313]{Yu} (see also \cite[Proposition 3.2]{qr10}), shows that $K_*(C_h)$ (up to canonical isomorphism) does not depend on the choice of standard representations $\rho_X$ and $\rho_Y$ of $C_0(X)$ and $C_0(Y)$, respectively.
\end{proof}

\begin{define}
\label{homotopiesofomor}
Let $[h:Y\to X]$ and $[h':Y'\to X]$ be continuous mappings. 
\begin{itemize}
\item We say that two morphisms $(f_j,g_j):[h:Y\to X]\to[h':Y'\to X'] $, $j=0,1$, are homotopic if there is a mapping 
$$(F,G):\left[\mathrm{id}_{[0,1]}\times h:[0,1]\times Y\to [0,1]\times X\right]\to \left[h':Y'\to X'\right],$$ 
such that $F(j,\cdot)=f_j$ and $G(j,\cdot)=g_j$ for $j=0,1$. 
\item If there are morphisms $u:[h:Y\to X]\to [h':Y'\to X']$ and $v: [h':Y'\to X']\to [h:Y\to X]$ such that $vu$ and $uv$ are homotopic to the identity mapping on $[h:Y\to X]$ and $[h':Y'\to X']$, respectively, we say that $[h:Y\to X]$ and $[h':Y'\to X']$ are homotopy equivalent.
\end{itemize}
\end{define}

Let $h:Y\to X$ be a continuous mapping of compact spaces. We define $\hat{X}$ as the mapping cylinder $[0,1]\times Y\cup_h \{0\}\times X$ and $\hat{h}:Y\to \hat{X}$ as the inclusion mapping $\hat{h}(y):=(0,y)$. The inclusion $X=X\times \{0\}\hookrightarrow \hat{X}$ fits into a commuting diagram
\begin{equation}
\label{cdwithxxhat}
\begin{CD}
Y@>h >> X   \\
@| @VVV   \\
Y @>\hat{h} >> \hat{X}.  \\
\end{CD}
\end{equation}
The projection mapping $\hat{X}\to X$ is a homotopy inverse of the inclusion $X=X\times \{0\}\hookrightarrow \hat{X}$. After some shorter considerations, one deduces the following result.

\begin{prop}
\label{homotopiesofomorcwy}
The maps in the diagram \eqref{cdwithxxhat} defines a homotopy equivalence between $[h:Y\to X]$ and $[\hat{h}:Y\to \hat{X}]$. In particular, for any compact pair $[h:Y\to X]$ there is a canonically homotopy equivalent pair $[\hat{h}:Y\to \hat{X}]$ where $\hat{h}$ is an inclusion and a canonical isomorphism $K_*^{CWY}(h)\cong K_*^{CWY}(\hat{h})$.
\end{prop}

We now turn to proving functoriality of the group $K_*^{CWY}(h)$. 

\begin{prop}
\label{functorcwy}
The group $K_*^{CWY}(h)$ depends functorially on $h$ in the following sense: if $(f,g):[h:Y\to X]\to [h':Y'\to X']$ is a continuous map of compact pairs, there is a canonically associated $(f,g)_*:K_*^{CWY}(h)\to K_*^{CWY}(h')$ and this association is (up to canonical isomorphisms) functorial for compositions. 
\end{prop}

We note that the association $(f,g)\mapsto (f,g)_*$ is not functorial in the strict sense.

\begin{proof}
We can by Proposition \ref{homotopiesofomorcwy} assume that $[h:Y\to X]$ and $[h':Y'\to X']$ are inclusion mappings. In particular, this assumption implicitly implies that $g=f|_Y$. By Proposition \ref{findepocwy} the groups $K_*^{CWY}(h)$ and $K_*^{CWY}(h')$ are independent on choice of standard representations, and we choose them as follows. Following the construction in \cite[Definition 2.5]{CWY}, we pick countable dense subsets $Z_Y\subseteq Y$, $Z_X\subseteq X$, $Z_{Y'}\subseteq Y'$ and $Z_{X'}\subseteq X'$ such that $Z_Y\subseteq Z_X$, $Z_{Y'}\subseteq Z_{X'}$, $f(Z_X)\subseteq Z_{X'}$ and $f(Z_Y)\subseteq Z_{Y'}$. Define $\rho_X$ as the representation defined from pointwise multiplication on $\ell^2(Z_X\times \Z)$. The representations $\rho_Y$, $\rho_{Y'}$ and $\rho_{X'}$ are defined analogously. The $*$-homomorphisms 
$$h:C^*_{L,\rho_Y}(Y)\to C^*_{L,\rho_X}(X)\quad\mbox{and}\quad h':C^*_{L,\rho_{Y'}}(Y')\to C^*_{L,\rho_{X'}}(X'),$$ 
are defined from the isometries induced from inclusion of the relevant sets. In particular, the maps $f$ and $g$ induces partial isometries $\ell^2(Z_X\times \Z)\to \ell^2(Z_{X'}\times \Z)$ and $\ell^2(Z_Y\times \Z)\to \ell^2(Z_{Y'}\times \Z)$ that defines $*$-homomorphisms $f$ and $g$ fitting into a commuting diagram 
$$\begin{CD}
C^*_{L,\rho_Y}(Y)@>h >> C^*_{L,\rho_X}(X)   \\
@VgVV @VVfV   \\
C^*_{L,\rho_{Y'}}(Y') @>h' >> C^*_{L,\rho_{X'}}(X').  \\
\end{CD}$$
We deduce that $f$ and $g$ defines a $*$-homomorphism $(f,g):C_h\to C_{h'}$ that fits into a commuting diagram
$$\begin{CD}
0@>>> SC^*_{L,\rho_X}(X) @>>> C_h @>>> C^*_{L,\rho_Y}(Y)@>>>0 \\
@.  @VSfVV  @V(f,g)VV  @VgVV \\
0@>>> SC^*_{L,\rho_{X'}}(X')@>>> C_{h'} @>>> C^*_{L,\rho_{Y'}}(Y') @>>>0 \\
\end{CD}$$
The map $(f,g)_*:K_*^{CWY}(h)\to K_*^{CWY}(h')$ is defined from the constructed $*$-homomorphism $(f,g)$. 

The association $(f,g)\mapsto (f,g)_*$ is clearly functorial in the sense that if if $(f,g):[h:Y\to X]\to [h':Y'\to X']$ and $(f',g'):[h:Y'\to X']\to [h'':Y''\to X'']$ are continuous maps of compact pairs then we can find models for the localization algebras as above giving rise to $*$-homomorphisms $(f,g):C_h\to C_{h'}$, $(f',g'):C_{h'}\to C_{h''}$ and $(f'f,g'g):C_h\to C_{h''}$ such that $(f'f,g'g)=(f',g')\circ (f,g)$.
\end{proof}

\begin{define}
\label{defcwynotation}
For locally compact spaces $X$ and $Y$, and a Lipschitz continuous $h:Y\to X$ we define 
\begin{align*}
K_*^{\textnormal{CWY}}(Y)&:=\varinjlim_{Y'\subseteq Y\mbox{\tiny compact}}K_*^{\textnormal{CWY}}(Y') \quad\mbox{and}\\ 
K_*^{\textnormal{CWY}}(h)&:=\varinjlim_{ X'\subseteq X,Y'\subseteq h^{-1}(X'), \mbox{\tiny compact}}K_*^{\textnormal{CWY}}(h|:Y'\to X').
\end{align*}
The compact case was defined in Equation \eqref{cwynotation}.
\end{define}

The definitions make sense because of Proposition \ref{functorcwy}. 

\begin{prop}
\label{cwyexact}
Let $h:Y\to X$ be a Lipschitz continuous mapping of metric locally compact spaces. The group $K_*^{\textnormal{CWY}}(h)$ fits into a six term exact sequence 
\begin{center}
$\begin{CD}
K_*^{\textnormal{CWY}}(Y) @>h_*>> K_*^{\textnormal{CWY}}(X) @>>> K_*^{\textnormal{CWY}}(h) \\
@AAA @.  @VVV  \\
K_{*-1}^{\textnormal{CWY}}(h) @<<<K_{*-1}^{\textnormal{CWY}}(X) @<h_*<< K_{*-1}^{\textnormal{CWY}}(Y)  \\
\end{CD}$
\end{center}
This six term exact sequence is compatible with the functoriality constructed in the proof of Proposition \ref{functorcwy}.
\end{prop}

\begin{proof}
For any compact $X'\subseteq X$ and a compact $Y'\subseteq h^{-1}(X')$ we have a pair $h|:Y'\to X'$ of compact pairs arriving at a short exact sequence
$$0\to SC^*_{L}(X')\to C_{h|:Y'\to X'} \to C^*_{L}(Y') \to 0.$$
After applying $K$-theory, we arrive at a six term exact sequence
\begin{center}
$\begin{CD}
K_*^{\textnormal{CWY}}(Y') @>h_*>> K_*^{\textnormal{CWY}}(X') @>>> K_*^{\textnormal{CWY}}(h|:Y'\to X') \\
@AAA @.  @VVV  \\
K_{*-1}^{\textnormal{CWY}}(h|:Y'\to X') @<<<K_{*-1}^{\textnormal{CWY}}(X') @<h_*<< K_{*-1}^{\textnormal{CWY}}(Y')  \\
\end{CD}$
\end{center}
The proposition follows from the fact that taking a direct limit defines an exact functor on the category of directed systems of abelian groups.
\end{proof}

\begin{remark}
\label{equi}
Returning to \cite[Definition 2.2]{CWY}, we also have the equivariant localization algebra $C^*_{L, {\rm max}}(\tilde{X})^{\pi_1(X)}$ and the equivariant Roe algebra $C^*_{{\rm max}}(\tilde{X})^{\pi_1(X)}$ defined in Subsection \ref{localsubs}. As above, we drop the ``max" from the notation. If $X$ is compact, then, since $\tilde{X}$ is equivariantly coarsely equivalent to $\pi_1(X)$ and \cite[Remark 3.14]{gongwangyu}, the algebra $C^*(\tilde{X})^{\pi_1(X)}$ is Morita equivalent to the maximal $C^*$-algebra of $\pi_1(X)$. 
\end{remark}

\begin{prop}(see \cite[page 7]{CWY})
\label{equiiso}
Let $X$ be a compact space. For any Galois covering $\Gamma\to \tilde{X}\to X$, there is an explicitly given asymptotic morphism $(\phi_s)_{s\in [1,\infty)}: \mathbb{R}_L(X) \rightarrow \mathbb{R}_L(\tilde{X})^{\Gamma}$ (see \cite[Section 2]{CWY}) which is continuous in the maximal $C^*$-norm for $s>>1$ inducing an isomorphism on $K$-theory:
\[(\phi_s)_*: K_*(C^*_L(X))\xrightarrow{\sim} K_*(C^*_L(\tilde{X})^{\Gamma}).\]
\end{prop}

The construction of the asymptotic morphism $(\phi_s)_{s\in [1,\infty)}$ can be found in \cite[Definition 2.9]{CWY}. It is constructed in a model for localization algebras similar to that used in the proof of Proposition \ref{functorcwy}. The reader should note that $(\phi_s)_{s\in [1,\infty)}$ is only densely defined, leaving a discrepancy to the usual definition of asymptotic morphisms \cite{connesbook,conneshigson}. However, having a densely defined asymptotic morphism continuous in the ambient $C^*$-norm suffices to define maps on $K$-theory. We omit the details and refer the reader to \cite[Lemma 2.10]{CWY}.

\begin{remark}
The reader should beware of the possible confusion that could arise from $\phi:\Gamma_1\to \Gamma_2$ denoting a group homomorphism for a large portion of the paper and the notation $(\phi_s)_{s\in [1,\infty)}$ for the asymptotic morphism from Proposition \ref{equiiso}. The two objects are not relateed. The reason for this notation is that $\phi:\Gamma_1\to \Gamma_2$ follows the notation from \cite{DeeMappingCone} and $(\phi_s)_{s\in [1,\infty)}$ the notation from \cite{CWY}. 
\end{remark}

Combining the isomorphism induced from the asymptotic morphism in Proposition \ref{equiiso} with the Morita equivalence from Remark \ref{equiiso}, we can define the assembly map:

\begin{define}
Let $X$ be a compact space and $\Gamma\to \tilde{X}\to X$ a Galois covering. The assembly map $\mu_{\rm CWY}:K_*^{\rm CWY}(X)\to K_*(C^*(\Gamma))$ is defined as the composition 
$$K_*^{\rm CWY}(X)\xrightarrow{(\phi_s)_*}K_*(C^*_L(\tilde{X})^{\Gamma})\xrightarrow{{\rm ev}_1}K_*(C^*(\tilde{X})^{\Gamma})\cong K_*(C^*(\Gamma)),$$
where the last isomorphism comes from the Morita equivalence in Remark \ref{equi}.
\end{define}

\begin{prop}(cf. \cite[page 9]{CWY}) 
A Lipschitz continuous mapping $h:Y \to X$, for a compact space $Y$, induces $*$-homomorphisms $h_{C^*}:C^*(\tilde{Y})^{\pi_1(Y)}\to C^*(\tilde{X})^{\pi_1(X)}$ and $h_L:C^*_L(\tilde{Y})^{\pi_1(Y)}\to C^*_L(\tilde{X})^{\pi_1(X)}$ compatible with the asymptotic morphism from Proposition \ref{equiiso} in the sense that $(\phi_s^X\circ h)_{s\in [1,\infty)}\sim (h_L\circ \phi_s^Y)_{s\in [1,\infty)}$, where the superscript indicate the involved space and $\sim$ denotes asymptotic equivalence.
\end{prop}

\begin{remark}
\label{staisoon}
The map from $C^*(\tilde{Y})^{\pi_1(Y)}$ to $C^*(\tilde{X})^{\pi_1(X)}$ can equivalently be defined as the composition
\[C^*(\tilde{Y})^{\pi_1(Y)} \cong C^*(\pi_1(Y))\otimes \mathcal{K} \rightarrow C^*(\pi_1(X))\otimes \mathcal{K} \cong C^*(\tilde{X})^{\pi_1(X)},\]
where $\mathcal{K}$ denotes the compact operators, the isomorphisms are the ones considered by Roe in \cite{Roe96}, and the map is given by the tensor product of a multiplicity preserving $*$-homomorphism $\mathcal{K}\to\mathcal{K}$ and the map induced from the group homomorphism $\phi: \pi_1(Y) \rightarrow \pi_1(X)$. This construction also leads to the map $C^*_L(\tilde{Y})^{\pi_1(Y)} \rightarrow C^*_L(\tilde{X})^{\pi_1(X)}$. In the notation used in \cite{CWY}, $h_{C^*}$ is denoted by $\psi$ and $h_L$ by $\psi_L$.
\end{remark}

We form the mapping cones $C_{h_{C^*}}$, $C_{h_L}$, and $C_{\phi}$. By using the mapping cone exact sequences we see that $K_*(C_{h_{C^*}}) \cong K_*(C_{\phi})$ (we always use the isomorphism constructed from the isomorphism $C^*(\tilde{X})^{\pi_1(X)} \cong C^*(\pi_1(X))\otimes \mathcal{K}$ used in Remark \ref{staisoon}).

\begin{define}(see \cite[page 9]{CWY}) 
\label{anaassmu}
Let $h:Y\to X$ be a Lipschitz continuous map of compact metric spaces. Define an asymptotic morphism $(\chi_s)_{s\in [1,\infty)} : C_h \rightarrow C_{h_L}$ via
$$ (f, a) \in C_h \mapsto (\phi^X_s(f),\phi^Y_s(a)).$$
The superscript indicates the spaces involved. The relative analytic assembly map $\mu_{{\rm max}}: K_*(SC_h) \to K_*(SC_\phi)$ is defined as the composition:
\[K_*(SC_h) \xrightarrow{(\chi_s)_*} K_*(SC_{h_L}) \xrightarrow{{\rm ev}_1} K_*(SC_{h_{C^*}})\cong K_*(SC_\phi).\]
\end{define}

In Definition \ref{anaassmu} we are abusing the notation as the asymptotic morphism $(\chi_s)_{s\in [1,\infty)}$ is not globally defined on $C_h$, it is in general only asymptotically defined on the dense $*$-subalgebra $(S\mathbb{R}_L(X)\oplus \mathbb{R}_L(Y))\cap C_h$. Again, by continuity of $(\chi_s)_{s\in [1,\infty)}$ in the maximal $C^*$-norms (for $s>>1$), a density argument ensures that the mapping on $K$-theory $(\chi_s)_*:K_*(SC_h) \to K_*(SC_{h_L})$ is well defined. The reader should also consult \cite{SchickSeyedhosseini} for an alternative approach to the relative assembly mapping.

Again, we often drop the ``max" from the notation and denote this map by $\mu$, but sometime denote this map by $\mu_{\phi}$ or $\mu_{CWY}$ to emphasize the mapping cone involved or to distinguish this analytic assembly map from the geometric assembly map defined in the next section. By construction and the functoriality of six term exact sequences, a direct limit argument ensures that the assembly mappings fit into a commuting diagram (see the proof of \cite[Theorem 2.17]{CWY}). 

\begin{prop}
If $h:Y\to X$ is a Lipschitz continuous mapping inducing the mapping $\phi:=\pi_1(h):\pi_1(Y)\to \pi_1(X)$, the relative assembly mapping fit into a commuting diagram with exact rows:
\tiny
\begin{center}
$\begin{CD}
@>>> K_*^{\textnormal{CWY}}(Y) @>h_*>> K_*^{\textnormal{CWY}}(X) @>>> K_*^{\textnormal{CWY}}(h) @>>> K_{*-1}^{\textnormal{CWY}}(Y) @>>> \\
@. @VV \mu_{{\rm max}} V @VV\mu_{{\rm max}}V  @VV\mu_{{\rm max}} V  @VV\mu_{{\rm max}} V\\
@>>> K_*(C^*(\pi_1(Y))) @>\phi >> K_*(C^*(\pi_1(X))) @>>> K_*(SC_\phi) @>>> K_{*-1}(C^*(\pi_1(Y))) @>>> \\
\end{CD}$
\end{center}
\normalsize
The upper row is the sequence from Proposition \ref{cwyexact} and the bottom row is the mapping cone sequence for $\phi$.
\end{prop}

Let us consider two relevant examples:

\begin{ex} 
For a compact manifold with boundary, $W$, we consider the inclusion mapping $i: \partial W \hookrightarrow \overline{W}$. In this example, $X=W$, $Y=\partial W$ and $h=i$. This mapping models all relative geometric $K$-homology groups; both at the level of relative $K$-homology of continuous mappings (see Subsection \ref{domainsubs} below) and relative $K$-theory of $*$-homomorphisms (see Subsection \ref{codomainsubs} below). The associated Roe algebras and localization algebras of this mapping were studied in detail in \cite{xieyu,zeid14}.
\end{ex}

\begin{ex} 
\label{itemtwo}
Associated with a group homomorphism $\phi:\Gamma_1\to \Gamma_2$ of finitely generated groups there is (after choosing basepoints) a unique homotopy class of a continuous mapping $B\phi : B\Gamma_1 \rightarrow B\Gamma_2$ such that $(B\phi)_*=\phi:\pi_1(B\Gamma_1)=\Gamma_1\to \pi_1(B\Gamma_2)=\Gamma_2$. As mentioned above, we fix a choice of $B\phi$.  We remind the reader of the standing assumption that $B\Gamma_j$ is a locally finite $CW$-complex ($j=1,2$). For the purpose of defining localization algebras and their functoriality we pick metrics $\mathrm{d}_j$ on $B\Gamma_j$ for $j=1,2$ defining the topologies and making $B\phi$ Lipschitz. After a choice of word metric, we can lift the metrics $\mathrm{d}_j$ to $E\Gamma_j$ in such a way that the $\Gamma_j$-action is isometric, each orbit is bi-Lipschitz equivalent to $\Gamma_j$ and the lifted $\Gamma_1$-equivariant mapping $E\phi:E\Gamma_1\to E\Gamma_2$ is Lipschitz. 
\end{ex}

\subsection{Localized index of Dirac operators}
\label{localindex}

The localized index $\mathrm{ind}_L:K_*(X)\to K_*(C^*_L(X))$ admits a simple description for Dirac operators. We write $K_*(X)=K_*^{\rm l.f.}(X):=KK_*(C_0(X),\field{C})$ where $X$ is a locally compact metric space. 

The localized index was defined in \cite{Yu} as follows. For an equivalent approach, see \cite{qr10} and \cite{SchickSeyedhosseini} for the extension of the techniques from \cite{qr10} to the maximal $C^*$-norm. As in the previous section, $\rho$ is a standard representation of $C_0(X)$ on a Hilbert space $\mathcal{H}$. An operator $T\in \mathbbm{B}(\mathcal{H})$ is said to be pseudo-local if the commutator $[T,\rho(a)]$ is compact for any $a\in C_0(X)$. The algebra $D^*_{\rho,{\rm max}}(X)$ denotes the maximal $C^*$-algebra generated by the finite propagation pseudo-local operators on $\mathcal{H}$. We often drop ${\rm max}$ and the representation $\rho$ from the notation and write $D^*(X)$. The $C^*$-algebra $D^*_L(X)$ is defined as the maximal $C^*$-algebra generated by all $g:[1,\infty)\to D^*(X)$ such that $\mathrm{prop}(g(s))\to 0$ as $s\to \infty$.

For each $n$, take a uniformly bounded locally finite cover $(U_{n,i})_i$ of $X$ with $\mathrm{diam}(U_{n,i})<1/n$ and a subordinate partition of unity $(\phi_{n,i})_i$. For a $K$-homology cycle $(H,F)$, we define $F_L\in D_L^*(X)$ by
$$F_L(s):=\sum_{i} (n-s)\sqrt{\phi_{n,i}}F\sqrt{\phi_{n,i}}+(n-s+1)\sqrt{\phi_{n+1,i}}F\sqrt{\phi_{n+1,i}}, \quad s\in [n-1,n].$$ 
Clearly, $F_L$ is a symmetry modulo $C^*_L(X)$ and hence it defines a class $[F_L\!\!\mod C^*_L(X)]\in K_{*+1}(D^*_L(X)/C^*_L(X))$.

\begin{define}[Localized index of $K$-homology cycles] 
\label{locaindedef}
The localized index $\mathrm{ind}_L:K_*(X)\to K_*(C^*_L(X))$ is defined as 
$$\mathrm{ind}_L(H,F):=\partial [F_L\!\!\!\mod C^*_L(X)],$$ 
where $\partial :K_{*+1}(D^*_L(X)/C^*_L(X))\to K_{*}(C^*_L(X))$ denotes the boundary mapping in $K$-theory. 
\end{define}

The mapping $\mathrm{ind}_L:K_*(X)\to K_*(C^*_L(X))$ is an isomorphism by \cite[Theorem 3.4]{qr10} if the $X$-module used to define $C^*_L(X)$ is \emph{very ample}, i.e., an infinite direct sum of standard representations.

\begin{remark}
\label{projectionapproach}
More explicitly, $\mathrm{ind}_L(H,F)=[P_F]-[1]$ where the projection $P_F\in M_2(C^*_L(X)\,\tilde{}\;)$ is defined by
\small
$$P_F(s):=
\begin{pmatrix}
F_L(s)F_L^*(s)+& &F_L(s)(1-F_L^*(s)F_L(s))+\\
(1-F_L(s)F_L^*(s))F_L(s)F_L^*(s)&&(1-F_L(s)F_L^*(s))F_L(s)(1-F_L^*(s)F_L(s))\\
& &\\
(1-F_L^*(s)F_L(s))F_L^*(s)&&(1-F_L^*(s)F_L(s))^2
\end{pmatrix}$$
\normalsize
This is described in more detail in \cite[Page 1392]{willettyu12}.
\end{remark}

We now turn to Dirac operators. Assume that $X$ is a complete Riemannian manifold. Let $S\to X$ be a Clifford bundle of bounded geometry and $D$ a Dirac operator on $S$. By \cite{chernoffess,roeasympt}, $D$ is self-adjoint on $L^2(X,S)$ when equipped with the domain given by the graph closure of $C^\infty_c(X,S)$ and the wave operator $\mathrm{exp}(itD)$ has propagation speed bounded by $c|t|$ for some $c=c(D)>0$. For any $\chi\in C^\infty(\field{R})$ such that $\chi^2-1\in C^\infty_0(\field{R})$ and the Fourier transform $\hat{\chi}$ has compact support, say $\mathrm{supp}(\hat{\chi})\subseteq [-R,R]$, \cite[Proposition 2.2]{Roe88} implies that $\chi(D)$ has finite propagation bounded by $cR$. In particular, $\mathrm{prop}(\chi(s^{-1}D))\leq cR/s$ so $(\chi(s^{-1}D))_{s\in [1,\infty)}\in D^*_L(X)$. By a density argument, $(\chi(s^{-1}D))_{s\in [1,\infty)}\in D^*_L(X)$ for any $\chi \in C_b(\field{R})$ with $\chi^2-1\in C_0(\field{R})$. Define $P_D\in M_2(C^*_L(X)\,\tilde{}\;)$ as in Remark \ref{projectionapproach} using $(\chi(s^{-1}D))_{s\in [1,\infty)}$ instead of $F_L$ for a $\chi\in C_b(\field{R})$ with $\lim_{u\to \pm \infty}\chi(u)=\pm 1$. 

\begin{prop}
If $D$ is a Dirac operator on a complete manifold, the localized index of $D$ given by
$$\mathrm{ind}_L(D):=[P_D]-[1]\in K_*(C^*_L(X)),$$
is well defined and independent of  $\chi\in C_b(\field{R})$ with $\lim_{u\to \pm \infty}\chi(u)=\pm 1$. Moreover,
$$\mathrm{ind}_L(D)=\partial[(\chi(s^{-1}D))_{s\in [1,\infty)}\!\!\mod C^*_L(X)]= \mathrm{ind}_L([D]),$$ 
where $[D]\in K_*(X)$ denotes the $K$-homology class on $C_0(X)$ associated with $D$. 
\end{prop}

\begin{proof}
It remains to prove that $\mathrm{ind}_L(D)=\mathrm{ind}_L([D])$. The desired identity follows once $(\chi(s^{-1}D))_{s\in [1,\infty)}-F_L\in C^*_L(X)$, in which case 
$$[(\chi(s^{-1}D))_{s\in [1,\infty)}\!\!\mod C^*_L(X)]=[F_L\!\!\mod C^*_L(X)]\in K_{*+1}(D^*_L(X)/C^*_L(X)).$$
Trivially, we have $(\chi(s^{-1}D))_{s\in [1,\infty)}-F_L\in D^*_L(X)$. We can assume that $\chi\in C^\infty(\field{R})$ satisfies $\chi^2-1\in C^\infty_c(\field{R})$ in which case, for every $s\in [1,\infty)$, $\chi(s^{-1}D)-F_L(s)$ is a pseudo-differential operator of order $-1$, hence locally compact. Therefore, $(\chi(s^{-1}D))_{s\in [1,\infty)}-F_L\in C^*_L(X)$.
\end{proof}

\subsection{Relative localized indices and mapping cones}
\label{relativeindexandmappingsconesubsec}
In fact, if $W$ is a Riemannian manifold with boundary the localized index of a Dirac operator can be considered an element in the $K$-theory of the mapping cone of localization algebras associated with the inclusion $i:\partial W\to \overline{W}$. We tacitly assume all geometric structures to be of product type near the boundary and follow the notation of Subsection \ref{notationseciont} (see page \pageref{notationseciont}).

\begin{define}[See page 314, \cite{Yu}]
If $X$ is a locally compact metric space and $Z\subseteq X$ is a closed subspace, the $C^*$-algebra $C^*_L(Z\subseteq X)$ is defined as the closure inside $C^*_L(X)$ of the elements $(T_s)_{s\in [1,\infty)}\in C^*_L(X)$ such that there is a function $\sigma:[1,\infty)\to \field{R}_{\geq 0}$ with $\lim_{s\to \infty}\sigma(s) =0$ and $\mathrm{d}_{X\times X}(\mathrm{supp}(T_s),Z\times Z)\leq \sigma(s)$.
\end{define}

We remark that $C^*_L(Z\subseteq X)\subseteq C^*_L(X)$ is an ideal and the inclusion $C^*_L(Z)\to C^*_L(Z\subseteq X)$ induces an isomorphism on $K$-theory by \cite[Lemma 3.10]{Yu}. 

\begin{prop}
Let $W$ be a compact manifold with boundary. 
\begin{enumerate}
\item Choose a $b\in C^\infty_c(W_2^\circ)$ such that $b=1$ on $\overline{W}_{3/2}$ and define the completely positive mapping $\tau_b:C^*_L(W_\infty)\to C^*_L(\overline{W}_2)$ by $\tau_b(T):=bTb$. 
\item Let $S_u:C^*_L(\overline{W}_2)\to C^*_L(\overline{W}_2)$ denote the $*$-homomorphism defined by shift by $u>0$ in the asymptotic parameter. 
\item Denote the quotient $C^*_L(\overline{W}_2)\to C^*_L(\overline{W}_2)/C^*_L([1,2]\times \partial W\subseteq \overline{W}_2)$ by $q$. 
\end{enumerate}

The asymptotic morphism 
$$(\beta^W_u)_{u\in [1,\infty)}: C^*_L(W_\infty)\to C^*_L(\overline{W}_2)/C^*_L([1,2]\times \partial W\subseteq \overline{W}_2)$$ 
defined as the composition 
$$C^*_L(W_\infty)\xrightarrow{\tau_b}C^*_L(\overline{W}_2)\xrightarrow{S_u} C^*_L(\overline{W}_2)\xrightarrow{q} C^*_L(\overline{W}_2)/C^*_L([1,2]\times \partial W\subseteq \overline{W}_2),$$
defines an isomorphism 
$$(\beta^W_u)_*: K_*(C^*_L(W_\infty))\xrightarrow{\sim} K_*(C^*_L(\overline{W}_2)/C^*_L([1,2]\times \partial W\subseteq \overline{W}_2)).$$ 
\end{prop}

\begin{proof}
For each $u$, $\beta^W_u$ is a completely positive mapping so we need to prove asymptotic multiplicativity. It suffices to prove the asymptotic properties on $\mathbb{R}_{L,\rho}(W_\infty)$. Take a $\tilde{b}\in C^\infty_c(W_2^\circ)$ such that $\tilde{b}=1$ on $\overline{W}_1$ and $\tilde{b}b=b$. If $T,T'\in \mathbb{R}_{L,\rho}(W_\infty)$ we write 
$$\tau_b(T)\tau_b(T')-\tau_b(TT')=bT(\tilde{b}-1)T'b+bT(b^2-\tilde{b})T'b.$$
Since $b^2-\tilde{b}\in C^\infty_c(W_2^\circ\setminus \overline{W})$, $S_u(bT(b^2-\tilde{b})T'b)\in C^*_L([1,2]\times \partial W\subseteq \overline{W}_2))$ for $u>>1$ large enough. Moreover, since $\mathrm{prop}(T'_s)\to 0$ as $s\to \infty$ it holds that $S_u(bT(\tilde{b}-1)T'b)=0$ for $u>>1$ large enough. Therefore 
$$q\circ S_u(\tau_b(T)\tau_b(T')-\tau_b(TT'))\to 0,\quad\mbox{as $u\to \infty$ uniformly in $s$.}$$ 
We conclude that $(\beta^W_u)_{u\in [1,\infty)}$ is an asymptotic morphism. That the mapping $(\beta^W_u)_{u\in [1,\infty)}$ induces an isomorphism on $K$-theory follows from the Mayer-Vietoris sequence \cite[Proposition 3.11]{Yu} and is explained in Theorem \ref{locindexthm} below.
\end{proof}

\begin{remark}
Let $\tilde{i}:C^*_L([1,2]\times \partial W\subseteq \overline{W}_2)\to C^*_L(\overline{W}_2)$ denote the inclusion. By the abstract properties of a mapping cone, there is a distinguished isomorphism $\alpha_W^0:K_*(C^*_L(\overline{W}_2)/C^*_L([1,2]\times \partial W\subseteq \overline{W}_2))\cong K_{*+1}(C_{\tilde{i}})$ (see \cite[Corollary 21.4.2]{blackbook}). Using a five lemma argument, the natural inclusion $C_i\subseteq C_{\tilde{i}}$, where $i:\partial W\to \overline{W}$ denotes the inclusion of the boundary, is an isomorphism on $K$-theory. We let 
$$\alpha_W:K_*(C^*_L(\overline{W}_2)/C^*_L([1,2]\times \partial W\subseteq \overline{W}_2))\cong K_{*+1}(C_i)$$ 
denote the composition of $\alpha_W^0$ with the isomorphism $K_*(C_{\tilde{i}})\cong K_*(C_i)$ induced from the inclusion $C_i\subseteq C_{\tilde{i}}$.
\end{remark}

\begin{define}
For a manifold with boundary $W$ and a Dirac type operator $D$ of product type near the boundary, we define 
$$\mathrm{ind}_L^{\rm rel}(D):=\alpha_W((\beta^W_u)_*\mathrm{ind}_L(D_\infty))\in K_{*+1}(C_i),$$
where $D_\infty$ denotes the extension of $D$ to $W_\infty$ using the cylindrical metric. Similarly, we define 
$$\mathrm{ind}_L^{\rm rel}:K_*(W^\circ)=K_*(W_\infty)\to K_{*+1}(C_i), \quad \mathrm{ind}_L^{\rm rel}:=\alpha_W\circ (\beta^W_u)_*\circ\mathrm{ind}_L, $$
where $\mathrm{ind}_L:K_*(W_\infty)\to K_*(C^*_L(W_\infty))$ is as in Definition \ref{locaindedef}.
\end{define}

\begin{theorem}
\label{locindexthm}
Let $W$ be a compact manifold with boundary and $i:\partial W\to W$ denote the inclusion of the boundary. The following diagram commutes, the rows are exact and the vertical arrows are isomorphisms:
\tiny
\begin{center}
$\begin{CD}
@>>> K_*(\partial W) @>i_*>> K_*(\overline{W}) @>>> K_*(W^\circ) @>\partial>> K_{*-1}(\partial W) @>>> \\
@. @VV\mathrm{ind}_LV @VV\mathrm{ind}_LV  @VV\mathrm{ind}_L^{\rm rel} V  @VV\mathrm{ind}_L V\\
@>>> K_*(C^*_L(\partial W)) @>i>> K_*(C^*_L(\overline{W})) @>>> K_{*-1}(C_i) @>>> K_{*-1}(C^*_L(\partial W)) @>>> \\
\end{CD}$
\end{center}
\normalsize
where the top row is associated with the short exact sequence $0\to C_0(W^\circ)\to C(\overline{W})\to C(\partial W)\to 0$ and the bottom row is the mapping cone sequence.
\end{theorem}

\begin{proof}
For notational simplicity, write $X_1:=\overline{W}$ and $X_2:=[1,\infty)\times \partial W$. We remark that $K_*(C^*_L(X_2))\cong K^*(C_0(X_2))\cong0$. Since $X_1\cup X_2=W_\infty$ and $X_1\cap X_2=\partial W$, the Mayer-Vietoris sequence \cite[Proposition 3.11]{Yu} (see also \cite[Corollary 5.3]{zeid14}) and the discussion in \cite[Proof of Theorem 3.2]{Yu}, gives us a commuting diagram with exact rows
\tiny
\begin{center}
$\begin{CD}
@>>> K_*(\partial W) @>i_*>> K_*(\overline{W}) @>>> K_*(W^\circ) @>\partial>> K_{*-1}(\partial W) @>>> \\
@. @V\mathrm{ind}_LVV @V\mathrm{ind}_LVV  @V\mathrm{ind}_LVV  @V\mathrm{ind}_LVV\\
@>>> K_*(C^*_L(\partial W)) @>i>> K_*(C^*_L(\overline{W})) @>>> K_{*}(C^*_L(W_\infty)) @>>> K_{*-1}(C^*_L(\partial W)) @>>> \\
\end{CD}$
\end{center}
\normalsize
The vertical arrows are isomorphisms by \cite[Theorem 3.4]{qr10}. On the other hand, the construction of the Mayer-Vietoris sequence implies that the following diagram is commutative with exact rows
\tiny
\begin{center}
$\begin{CD}
@>>> K_*(C^*_L(\partial W)) @>i>> K_*(C^*_L(\overline{W})) @>>> K_{*}(C^*_L(W_\infty)) @>>> K_{*-1}(C^*_L(\partial W)) @>>> \\
@. @V=VV @V=VV  @V\alpha_W\circ (\beta_u)_*VV  @V=VV\\
@>>> K_*(C^*_L(\partial W)) @>i>> K_*(C^*_L(\overline{W})) @>>> K_{*-1}(C_i) @>>> K_{*-1}(C^*_L(\partial W)) @>>> \\
\end{CD}$
\end{center}
\normalsize
Combining these two diagrams, we arrive at the commutativity of the diagram in the statement of the theorem. Finally, $\mathrm{ind}_L^{\rm rel}:=\alpha_W\circ (\beta^W_u)_*\circ\mathrm{ind}_L$ is an isomorphism because of a five lemma argument and \cite[Theorem 3.4]{qr10}.
\end{proof}

\section{The relative assembly map as a geometrically defined map} 
\label{GeoMap}

We consider a geometric version of the analytic assembly map considered in Definition \ref{anaassmu}. 
To define the required map, we must introduce the geometric cycles used to define the domain and codomain; we begin with the domain and proceed with the codomain. This section is independent of the results in the previous section. However, the relationship between these two constructions will be considered in detail in Sections \ref{IsoGeoToAna} and \ref{comAssMap}. The term ``cycle" refers to various objects. However, context and notation should make clear which definition of ``cycle" (e.g., Definitions \ref{geoCycRelClassSpace}, \ref{geoCycMappingCone}, etc) is being used.

\subsection{The domain - relative $K$-homology}
\label{domainsubs}
Throughout this subsection we fix a continuous mapping $h:Y\to X$. The domain of the relative assembly map is defined using cycles of the following form in the case $h=B\phi$ where $B\phi$ is obtained from a group homomorphism $\phi:\Gamma_1\to \Gamma_2$ (see the discussion in the introduction).

\begin{define}[see \cite{BDT}] 
\label{geoCycRelClassSpace}
A geometric cycle with respect to $h:Y\to X$ is a triple $(W, E, (f,g))$ where
\begin{enumerate}
\item $W$ is a compact, smooth, spin$^c$-manifold with boundary;
\item $E$ is a smooth complex vector bundle over $W$;
\item $f$ is a continuous map from $W$ to $X$, $g$ is a continuous map from $\partial W$ to $Y$, and  $ f|_{\partial W} = h \circ g$, i.e. $(f,g):[i:\partial W\to W]\to [h:Y\to X]$.
\end{enumerate}
\end{define}

As in standard geometric $K$-homology \cite{BD, BDT}, $W$ need not be connected and there is a natural $\field{Z}/2$-grading on cycles defined using the dimensions of the connected components of $W$ modulo two. Furthermore, there is a definition of isomorphism for cycles and when we refer to a ``cycle" we mean ``an isomorphism class of a cycle". The opposite of a cycle is given by the same cycle but with the opposite spin$^c$-structure on the manifold; given a spin$^c$-manifold, $W$, we denote the same manifold with the opposite spin$^c$-structure by $-W$. The set of cycles form an abelian semi-group under the disjoint union operation. On the semigroup of cycles, there is a disjoint union/direct sum relation, bordism relation and a relation coming from vector bundle modification. The reader can find more on these concepts in (for example) \cite{BD}; we will give a detailed development for the latter two aspects (i.e., bordism and vector bundle modification).

\begin{define}
A regular domain, $M_0$, of a manifold $M$ is a closed subset of $M$ which has nonempty interior and satisfies the following condition: if $x \in \partial M_0$, then there exists a coordinate chart, $\phi: U \rightarrow \field{R}^n$ centered at $x$ such that 
$$\phi(M_0 \cap U) = \{ (z_1, \ldots, z_n) \in \phi(U) \subseteq \field{R}^n \: | \: z_n \ge 0\}$$
\end{define}

\begin{define}
\label{bordisminrelkhm}
A bordism or a cycle with boundary with respect to $h:Y\to X$ is a collection $((Z, W), F, (h_2, h_1))$ where
\begin{enumerate}
\item $Z$ and $W$ are compact, smooth, spin$^c$-manifolds with boundary;
\item $W$ is a regular domain in $\partial Z$;
\item $F$ is a smooth complex vector bundle over $Z$;
\item $h_2: Z \rightarrow X$ is a continuous map, $h_1: \partial Z - {\rm int}(W) \rightarrow Y$ is a continuous map, and $ h_2|_{\partial Z - {\rm int}(W)} = h \circ h_1$.
\end{enumerate} 
The boundary of $((Z, W), F, (h_2, h_1))$ is \emph{defined to be} $(W, F|_W, (h_2|_W, h_1|_{\partial W}))$ (one can check that it is a cycle). Finally, two cycles are bordant if the disjoint union of the first with the opposite of the second is a boundary.
\end{define}

\begin{define}
\label{vectormod}
Let $(W, E, (f,g))$ be a cycle and $V$ be a spin$^c$-vector bundle with even dimensional fibers over it. We define the vector bundle modification of $(W, E, (f,g))$ by $V$ to be $(S(V\oplus {\bf 1}_\field{R}), \pi^*(E) \otimes B, (f\circ p, g\circ p))$ where
\begin{enumerate}
\item ${\bf 1}_\field{R}\to W$ denotes the trivial real line bundle;
\item $p:S(V\oplus {\bf 1}_\field{R})\to W$ is the bundle projection on the sphere bundle of $V\oplus {\bf 1}_\field{R}$;
\item $B\to S(V\oplus {\bf 1}_\field{R})$ is the Bott bundle (see \cite{BD}).
\end{enumerate}
We denote the cycle so obtained by $(W, E, (f,g))^V$; one can verify that the result of this process is a cycle.
\end{define}

\begin{define}
\label{defofkhomgeorel}
Let $K_*^{\textnormal{geo}}(h):=\{(W, E, (f,g))\}/\sim$ where $\sim$ is the equivalence relation generated by the disjoint union/direct sum relation, bordism and vector bundle modification.
\end{define}

\begin{remark}
\label{firstorientedremark}
Geometric $K$-homology can equivalently be modelled on oriented manifolds, see \cite{guentnerkhom,Kescont}. The same holds for the relative groups. In the oriented model, cycles are given by triples $(W,E,(f,g))$ where $W$ is a compact oriented manifold with boundary, $E$ is a Clifford bundle on $W$ and $(f,g)$ is as before. The isomorphism between the models is constructed as in \cite[Lemma 2.8]{Kescont}.
\end{remark}

The previous definition along with the next theorem are well-known. The notion of ``normal bordism" (see any of \cite{BCW, DeeZkz, DeeRZ, DG, Jak, Rav}) can be used to prove the next theorem. In fact, the proofs in \cite{Jak} and \cite[Proposition 4.6.8]{Rav} generalize to this situation with little change; we therefore omit the details of the proof. Note that the special case when $h$ is an inclusion of a closed subspace was the first case considered, see \cite{BDrelCstar, BDT}.

\begin{theorem} \label{relGeoKhomDomExaSeq}
The following sequence is exact:
\begin{center}
$\begin{CD}
K_0^{\textnormal{geo}}(Y) @>h_*>> K_0^{\textnormal{geo}}(X) @>r>> K_0^{\textnormal{geo}}(h) \\
@AA\delta A @. @VV\delta V \\
K_1^{\textnormal{geo}}(h) @<r<<  K_1^{\textnormal{geo}}(X) @<h_*<< K_1^{\textnormal{geo}}(Y) 
\end{CD}$
\end{center}
where the maps are defined as follows
\begin{enumerate}
\item $h_*$ is the map on $K$-homology induced from $h$; it is defined at the level of cycle via $(M,E,f) \mapsto (M,E,h \circ f)$;
\item $r$ is defined at the level of cycles via $(M,E,f) \mapsto (M,E,(f,\emptyset))$;
\item $\delta$ is defined at the level of cycles via $(W,E,(f,g)) \mapsto (\partial W, E|_{\partial W}, g)$.
\end{enumerate}
\end{theorem}

\begin{remark}
\label{universalmaps}
If $Z$ is a manifold with boundary, there are mappings $f_Z:Z\to B\pi_1(Z)$ and $f_{\partial Z}:\partial Z\to B\pi_1(\partial Z)$ fitting into a commuting diagram
$$\begin{CD}
\partial Z @>i>> Z\\
@VVf_{\partial Z} V  @VVf_Z V \\
B\pi_1(\partial Z) @>B\phi >> B\pi_1(Z),
\end{CD}$$
where $i:\partial Z\hookrightarrow Z$ denotes the inclusion and $\phi:=i_*:\pi_1(\partial Z)\to \pi_1(Z)$. Indeed, the universal property of classifying spaces guarantee the existence of $f_Z$ and $f_{\partial Z}$ making the diagram commute \emph{up to homotopy} and the homotopy extension principle can be used to construct functions making the diagram commute. We write this as $(f_Z,f_{\partial Z}):[i:\partial Z\to  Z]\to [Bi: B\pi_1(\partial Z)\to B\pi_1( Z)]$. The functoriality of relative $K$-homology gives rise to a mapping 
\begin{align*}
(f_Z,f_{\partial Z})_*:K_*^{\rm geo}(i)&\to K_*^{\rm geo}(B\phi), \\
&(W,E,(f,g))\mapsto (W,E,(f_Z\circ f,f_{\partial Z}\circ g)).
\end{align*}
\end{remark}

\begin{remark}
\label{injectiveh}
Let $h:Y\to X$ be a Lipschitz mapping of compact metric spaces inducing a $*$-homomorphism $h:C^*_L(Y)\to C^*_L(X)$. If we assume that $h$ is a homeomorphism onto its range, the following mapping is well defined 
$$\mathrm{ind}_L^{\rm rel}:K_*^{\rm geo}(h)\to K_{*}^{\rm CWY}(h), \quad (W,E,(f,g))\mapsto (f,g)_*\mathrm{ind}_L^{\rm rel}(D_E),$$ 
where $D_E$ is a Dirac operator on $S_W\otimes E$. In the case that $h$ is a homeomorphism onto its range, $Y\cong h(Y)\subseteq X$ is closed and the fact that $\mathrm{ind}_{L,h}^{\rm rel}$ is well defined follows from that it factors as a mapping from cycles to classes over the analytic assembly mapping $K_*^{\rm geo}(h)\to K_*(X\setminus h(Y))$ (see \cite[Theorem 6.1]{BHS}) and the localized index $\mathrm{ind}_L:K_*(X\setminus h(Y))\to K_*(C^*_L(X\setminus h(Y)))$. 

Up to homotopy, a general mapping $h:Y\to X$ can be realized as a homeomorphism onto its range. Using Proposition \ref{homotopiesofomorcwy} (see page \pageref{homotopiesofomorcwy}) and some additional diagram chases, one can prove that the relative localized index induces a well defined  map on relative $K$-homology for a general map $h$. We will proceed with a more direct, geometric approach.
\end{remark}

We will soon define the localized index at the level of relative $K$-homology for a general mapping $h$. This is done easily with Theorem \ref{locindexthm} and the next lemma at hand. 

\begin{lemma}
\label{lemmaglue}
Let $h:Y\to X$ be a Lipschitz mapping of locally compact metric spaces inducing a $*$-homomorphism $h:C^*_L(Y)\to C^*_L(X)$. We assume that $(W,E,(f,g))$ and $(W',E',(f',g'))$ are cycles for $K_*^{\rm geo}(h)$ such that $\partial W=-\partial W'$, $g=g'$ and $E|_{\partial W}=E'|_{\partial W'}$ and thus obtaining a cycle $(M,\hat{E},\hat{f})$ for $K_*^{\rm geo}(X)$ defined from the closed spin$^c$-manifold $M:=W\cup_{\partial W}W'$, the vector bundle $\hat{E}:=E\cup_{\partial W}E'\to M$ and the mapping $\hat{f}:=f\cup_{\partial W} f':M\to X$. Then 
$$ (f,g)_*\mathrm{ind}_L^{\rm rel}(D_E^W)+(f',g')_*\mathrm{ind}_L^{\rm rel}(D_{E'}^{W'})=j_*\hat{f}_*\mathrm{ind}_L(D_{\hat{E}}^{M}),$$
where $D_{\hat{E}}^{M}$ is a Dirac operator on $M$ and $j_*:K_*^{\rm CWY}(X)=K_*(C^*_L(X))\to K_*^{\rm CWY}(h)=K_{*+1}(C_h)$ denotes the canonical mapping.
\end{lemma}

\begin{proof}
We can by functoriality reduce to the case that $X=M$, $Y=\partial W$ and $h$ is the inclusion $\partial W\to M=W\cup_{\partial W}W'$. We can moreover assume that $f:W\to M=W\cup_{\partial W}W'$ and $f':W'\to M=W\cup_{\partial W}W'$ are the inclusions, in which case $g=f|_{\partial W}$ and $g'=f'|_{\partial W'}$. By Remark \ref{injectiveh}, the mapping
$$\mathrm{ind}_L^{\rm rel}:K_*^{\rm geo}(h:\partial W\to M)\to K_*^{\rm CWY}(h:\partial W\to M)$$ 
is well defined, and it thus suffices to prove the existence of a bordism 
$$(W,E,(f,g))\dot{\cup}(W',E',(f',g'))\dot{\cup} (-M,\hat{E},(\hat{f},\emptyset))\sim_{\rm bor}0 \quad\mbox{in $K_*^{\rm geo}(h:\partial W\to M)$.}$$

We consider the manifold with boundary 
$$Z:=[0,1]\times (W\cup_{\partial W}([0,1]\times \partial W)\cup_{\partial W} W').$$ 
Since $M\cong W\cup_{\partial W}([0,1]\times \partial W)\cup_{\partial W} W'$, we can identify $\partial Z$ with $\{0\}\times M\dot{\cup}\{1\}\times (-M)$ and $\hat{W}:=\{0\}\times (W\dot{\cup} W')\dot{\cup}\{1\}\times (-M)\subseteq \partial Z$ with a regular domain. We have $\partial Z\setminus \hat{W}=[0,1]\times M$. We define the mappings $h_1:\partial Z\setminus \hat{W}\to M$ and $h_2:Z=[0,1]\times M\to M$ as the projections and the vector bundle $F:=h_2^*(E\cup_{\partial W} E')$. We obtain a cycle with boundary $((Z,\hat{W}),F,(h_2,h_1))$ for $K_*^{\rm geo}(h:\partial W\to M)$, and its boundary coincides with $(W,E,(f,g))\dot{\cup}(W',E',(f',g'))\dot{\cup} (-M,\hat{E},(\hat{f},\emptyset))$.
\end{proof}

Recall the notation in Definition \ref{defcwynotation} on page \pageref{defcwynotation}.

\begin{theorem}
\label{relaassonkhom}
Let $h:Y\to X$ be a Lipschitz mapping of locally compact metric spaces inducing a $*$-homomorphism $h:C^*_L(Y)\to C^*_L(X)$. Then the following mapping is well defined 
$$\mathrm{ind}_L^{\rm rel}:K_*^{\rm geo}(h)\to K_{*}^{\rm CWY}(h), \quad (W,E,(f,g))\mapsto (f,g)_*\mathrm{ind}_L^{\rm rel}(D_E),$$ 
where $D_E$ is a Dirac operator on $S_W\otimes E$. The mapping $\mathrm{ind}_L^{\rm rel}$ fits into a commuting diagram with exact rows
\tiny
\begin{center}
$\begin{CD}
@>>> K_*^{\rm geo}(Y) @>h_*>> K_*^{\rm geo}(X) @>r>> K_*^{\rm geo}(h) @>\delta>> K_{*-1}(Y) @>>> \\
@. @VV\mathrm{ind}_LV @VV\mathrm{ind}_LV  @VV\mathrm{ind}_L^{\rm rel} V  @VV\mathrm{ind}_L V\\
@>>> K_*^{\rm CWY}(Y) @>h_*>> K_*^{\rm CWY}(X) @>>> K_{*}^{\rm CWY}(h) @>>> K_{*-1}^{\rm CWY}(Y) @>>> \\
\end{CD}$
\end{center}
\normalsize
If $X$ and $Y$ are locally finite $CW$-complexes, the mapping $\mathrm{ind}_L^{\rm rel}$ is an isomorphism.
\end{theorem} 

\begin{proof}
Assuming the map $\mathrm{ind}_L^{\rm rel}$ is well defined, functoriality implies that the proof that the diagram commutes reduces to the case when $h$ is the inclusion of the boundary of a manifold. In this particular case, Theorem \ref{locindexthm} implies that the diagram commutes. 

As such, we need only show that $\mathrm{ind}_L^{\rm rel}$ is well defined. For simplicity, we assume that $X$ and $Y$ are compact. To emphasize the $h$-dependence we write $\mathrm{ind}_{L,h}^{\rm rel}$ throughout the proof. As noted above in Remark \ref{injectiveh}, $\mathrm{ind}_{L,h}^{\rm rel}$ is well defined when $h$ is a homeomorphism onto its range. 

For a general $h$, we proceed by proving that $\mathrm{ind}_{L,h}^{\rm rel}$ respects the relations defining $K_*^{\rm geo}(h)$ (see Definition \ref{defofkhomgeorel} on page \pageref{defofkhomgeorel}). It is immediate that $\mathrm{ind}_{L,h}^{\rm rel}$ respects the disjoint union/direct sum relation. 

To prove that $\mathrm{ind}_{L,h}^{\rm rel}$ respects vector bundle modification, we consider a cycle $(W,E,(f,g))$ for $K_*^{\rm geo}(h)$. Let $i:\partial W\to W$ denote the inclusion of the boundary. Remark \ref{injectiveh} implies that $\mathrm{ind}_{L,i}^{\rm rel}$ is well defined, and therefore respects vector bundle modification. In particular, for a spin$^c$-vector bundle $V\to W$ of even rank, 
$$\mathrm{ind}^{\rm rel}_{L,i}(W,E,(\mathrm{id}_W,\mathrm{id}_{\partial W}))=\mathrm{ind}^{\rm rel}_{L,i}(W,E,(\mathrm{id}_W,\mathrm{id}_{\partial W})^V).$$
By functoriality, 
\begin{align*}
\mathrm{ind}^{\rm rel}_{L,h}(W,E,(f,g))&=(f,g)_*\mathrm{ind}^{\rm rel}_{L,i}(W,E,(\mathrm{id}_W,\mathrm{id}_{\partial W}))=\\
&=(f,g)_*\mathrm{ind}^{\rm rel}_{L,i}(W,E,(\mathrm{id}_W,\mathrm{id}_{\partial W})^V)=\mathrm{ind}^{\rm rel}_{L,h}(W,E,(f,g)^V),
\end{align*}
and $\mathrm{ind}^{\rm rel}_{L,h}$ respects vector bundle modification. 

We prove bordism invariance of $\mathrm{ind}^{\rm rel}_{L,h}$ using Lemma \ref{lemmaglue}. If $((Z,W),F,(h_2,h_1))$ is a cycle with boundary for $K_*^{\rm geo}(h)$, Lemma \ref{lemmaglue} implies that 
\begin{equation}
\label{glueingapart}
j_*(h_2)_*\mathrm{ind}_L(D_{F}^{Z})= (h_2|_W,h_1|_{\partial W})_*\mathrm{ind}_L^{\rm rel}(D_F^W)+(h_2|_{\partial Z\setminus W^\circ},h_1|_{\partial W})_*\mathrm{ind}_L^{\rm rel}(D_{F}^{\partial Z\setminus W^\circ}).
\end{equation}
By bordism invariance of $\mathrm{ind}_L:K_*^{\rm geo}(X)\to K_*^{\rm CWY}(X)$, $(h_2)_*\mathrm{ind}_L(D_{F}^{Z})=0$ in $K_*^{\rm CWY}(X)=K_*(C^*_L(X))$. Therefore, the left hand side of Equation \eqref{glueingapart} vanishes. As for the second term in the right hand side of Equation \eqref{glueingapart}, since $h_1$ is defined on $\partial Z\setminus W^\circ$ and $h_2|_{\partial Z\setminus W^\circ}=h\circ h_1$, it comes from the cycle $(\partial Z\setminus W^\circ, F_{\partial Z\setminus W^\circ},(h_1,h_1|_{\partial W}))$ for $K_*^{\rm geo}(\mathrm{id}_Y)$. More precisely, it can be written as 
$$(h_2|_{\partial Z\setminus W^\circ},h_1|_{\partial W})_*\mathrm{ind}_L^{\rm rel}(D_{F}^{\partial Z\setminus W^\circ})=(h,\mathrm{id}_Y)_*(h_1,h_1|_{\partial W})_*\mathrm{ind}_L^{\rm rel}(D_{F}^{\partial Z\setminus W^\circ}).$$
Here $(h_1,h_1|_{\partial W})_*\mathrm{ind}_L^{\rm rel}(D_{F}^{\partial Z\setminus W^\circ})\in K_*^{\rm CWY}(\mathrm{id}_Y)=0$. Therefore, the second term of the right hand side of Equation \eqref{glueingapart}, $(h_2|_{\partial Z\setminus W^\circ},h_1|_{\partial W})_*\mathrm{ind}_L^{\rm rel}(D_{F}^{\partial Z\setminus W^\circ})\in K_*^{\rm CWY}(h)$, vanishes. We conclude that 
$$\mathrm{ind}^{\rm rel}_{L,h}(W,F|_W,(h_2|_W,h_1|_{\partial W}))=(h_2|_W,h_1|_{\partial W})_*\mathrm{ind}_L^{\rm rel}(D_F^W)=0,$$ 
and $\mathrm{ind}^{\rm rel}_L$ is bordism invariant. 

\end{proof}

\subsection{The codomain - geometric relative $K$-theory of $\phi$} 
\label{codomainsubs}
We now turn to a geometric model for the mapping cone of $\phi:C^*(\Gamma_1)\to C^*(\Gamma_2)$. This model was studied in detail in \cite{DeeRZ} and applies in the more general situation of any unital $*$-homomorphism between unital $C^*$-algebras.

\begin{define}(see \cite[Definition 4.1]{DeeRZ}) 
\label{geoCycMappingCone} 
\\
A geometric cycle with respect to a unital $*$-homomorphism $\phi: C^*(\Gamma_1) \rightarrow C^*(\Gamma_2)$ is $(W, (E_{C^*(\Gamma_2)}, F_{C^*(\Gamma_1)}, \alpha))$ where
\begin{enumerate}
\item $W$ is a smooth, compact spin$^c$-manifold with boundary;
\item $E_{C^*(\Gamma_2)}$ is a $C^*(\Gamma_2)$-bundle over $W$;
\item $F_{C^*(\Gamma_1)}$ is a $C^*(\Gamma_1)$-bundle over $\partial W$;
\item $\alpha: E_{C^*(\Gamma_2)}|_{\partial W} \rightarrow E_{C^*(\Gamma_1)}\otimes_{\phi} C^*(\Gamma_2)$ is an isomorphism of $C^*(\Gamma_2)$-bundles.
\end{enumerate}
\end{define}

Results in \cite{DeeRZ, DeeMappingCone} imply that such cycles can be arranged (by moding out by an geometrically defined equivalence relation) into an abelian group, $K_*^{\textnormal{geo}}(pt;\phi)$. The relation is generated by disjoint union/direct sum relation, bordism and vector bundle modification in a manner similar to the ones defined above in Definitions \ref{bordisminrelkhm} and \ref{vectormod}. For the precise definitions in the context of this model see \cite[Section 4]{DeeRZ}. The group $K_*^{\textnormal{geo}}(pt;\phi)$ forms a realization of the Kasparov group $KK^*(\field{C},SC_{\phi})$. The interested reader can find further details on this model in \cite{DeeRZ, DeeMappingCone}. A fundamental property of this construction is the following theorem:

\begin{theorem}[special case of Theorem 4.21 in \cite{DeeRZ}]
\label{exactonk}
The following sequence is exact:
\begin{center}
$\begin{CD}
K_0^{\textnormal{geo}}(pt; C^*(\Gamma_1))) @>\phi_*>> K_0^{\textnormal{geo}}(pt; C^*(\Gamma_2))) @>r>> K_0^{\textnormal{geo}}(pt; \phi) \\
@AA\delta A @. @VV\delta V \\
K_1^{\textnormal{geo}}(pt; \phi) @<r<<  K_1^{\textnormal{geo}}(C^*(\Gamma_2)) @<\phi_*<< K_1^{\textnormal{geo}}(pt; C^*(\Gamma_1)) 
\end{CD}$
\end{center}
where the maps are defined at the level of cycles via
\begin{enumerate}
\item $r ( M, E_{C^*(\Gamma_2)}):= ( M, (E_{C^*(\Gamma_2)}, \emptyset, \emptyset))$
\item $\delta(W, (E_{C^*(\Gamma_2)}, F_{C^*(\Gamma_1)}, \alpha)):= (\partial W, F_{C^*(\Gamma_1)})$ 
\end{enumerate}
\end{theorem}

The abelian group $K_*^{\textnormal{geo}}(pt;\phi)$ is the codomain of the assembly map, which is the main topic of this section. The analytic realization of cycles in $K_*^{\textnormal{geo}}(pt;\phi)$ is more complicated than the approach using localization algebras in $K$-homology. We address this issue below in Section \ref{IsoGeoToAna}.

\begin{remark}
\label{secondorientedremark}
Also the group $K_*^{\rm geo}(pt;\phi)$ can be modelled on oriented manifolds as in \cite{guentnerkhom,Kescont}, cf. Remark \ref{firstorientedremark}. For $C^*$-coefficients $A$ one uses $A$-Clifford bundles as in \cite[Definition 2.9]{DGII}. The reader is referred to \cite[Section 2.3]{DGII} for the isomorphism of the oriented model with the spin$^c$-model for geometric $K$-homology with $C^*$-algebra coefficients.

\end{remark}

\subsection{The geometric relative assembly map}
\label{geoassemblysubsec}

With the definition of the domain and codomain of map developed, we can define a geometric analogue of the relative assembly map. As above, $\phi: \Gamma_1 \rightarrow \Gamma_2$ denotes a group homomorphism and we also let $\phi$ denote its induced map on the full group $C^*$-algebras. The Mishchenko bundle $\mathcal{L}_{B\Gamma}:=E\Gamma\times_\Gamma C^*(\Gamma)\to B\Gamma$ of a discrete group $\Gamma$ was defined in Equation \eqref{mishdef} (see page \pageref{mishdef}). 

We start by comparing the bundles $\mathcal{L}_{B\Gamma_1}$ and $\mathcal{L}_{B\Gamma_2}$. From the discussion in Subsection \ref{notationseciont}, we can assume that $B\phi$ fits into the following commutative diagram:
\begin{equation}
\label{comdipsix}
\begin{CD}
E\Gamma_1 @>E \phi >> E\Gamma_2 \\
@VVp_1 V  @VVp_2 V \\
B\Gamma_1 @>B\phi >> B\Gamma_2 
\end{CD}
\end{equation}
where $E\phi$ and $B\phi$ are the continuous maps induced from $\phi$, $p_1$ and $p_2$ are the projection maps. Furthermore, the diagram intertwines the group actions of $\Gamma_1$ and $\Gamma_2$.

\begin{prop} 
\label{MisBunProp}
Given a $\Gamma_1$-equivariant lift $E\phi$ of $B\phi$ as in the diagram \eqref{comdipsix}, there exists an explicit isomorphism $\alpha_0 : \mathcal{L}_{B\Gamma_1} \otimes_{\phi} C^*(\Gamma_2) \rightarrow (B\phi)^*(\mathcal{L}_{B\Gamma_2})$.
\end{prop}

\begin{proof}
We have the explicit identification
$$
\mathcal{L}_{B\Gamma_1} \otimes_{C^*(\Gamma_1)} C^*(\Gamma_2) =(E\Gamma_1\times_{\Gamma_1}C^*(\Gamma_1))\otimes_\phi C^*(\Gamma_2) \cong  E\Gamma_1 \times_{\phi} C^*(\Gamma_2),$$
where $E\Gamma_1 \times_{\phi} C^*(\Gamma_2)$ is the quotient of $E\Gamma_1 \times C^*(\Gamma_2)$ by the equivalence relation $(\gamma x,v)\sim (x,\phi(\gamma)v)$ for $x\in E\Gamma_1$, $v\in C^*(\Gamma_2)$ and $\gamma\in \Gamma_1$. We can identify
$$(B\phi)^*(\mathcal{L}_{B\Gamma_2}) \cong  (B\Gamma_1 \times_{B\Gamma_2} E\Gamma_2) \times_{\Gamma_2} C^*(\Gamma_2).$$
Let $\Psi:  E\Gamma_1 \times_{\phi} C^*(\Gamma_2) \rightarrow (B\Gamma_1 \times_{B\Gamma_2} E\Gamma_2) \times_{\Gamma_2} C^*(\Gamma_2)$ be the isomorphism of $C^*(\Gamma_2)$-bundles defined via
$$[z,v] \mapsto [\kappa(z), v]$$
where $\kappa:=p_1\times E\phi: E\Gamma_1 \rightarrow B\Gamma_1 \times_{B\Gamma_2} E \Gamma_2$. The map $\kappa$ is well defined due to the pullback diagram \eqref{comdipsix} and $\Psi$ is readily verified to be an isomorphism.
\end{proof}

\begin{define}
\label{defofassembly}
The geometric assembly map $\mu_{\phi}:K_*^{\textnormal{geo}}(B\phi) \rightarrow K_*^{\textnormal{geo}}(pt;\phi)$ is defined at the level of cycles via
$$(W,E,(f,g) ) \mapsto (W, (E \otimes_{\field{C}}f^*(\mathcal{L}_{B\Gamma_2}) ,E|_{\partial W} \otimes_{\field{C}}g^*(\mathcal{L}_{B\Gamma_1}) ,\alpha))$$
where $\alpha$ is the isomorphism of $C^*(\Gamma_2)$-bundles:
\begin{eqnarray*}
(E \otimes_{\field{C}}f^*(\mathcal{L}_{B\Gamma_2}))|_{\partial W} & \cong & E|_{\partial W} \otimes_{\field{C}}(f|_{\partial W})^*(\mathcal{L}_{B\Gamma_2}) \\
& \cong & E|_{\partial W} \otimes_{\field{C}}g^* ( (B\phi)^*( \mathcal{L}_{B\Gamma_2})) \\
& \xrightarrow{\mathrm{id}_{E|_{\partial W}}\otimes g^*\alpha_0} & E|_{\partial W} \otimes_{\field{C}}g^*(\mathcal{L}_{B_1} \otimes_{\phi} C^*(\Gamma_2)) \\
& \cong & (E|_{\partial W} \otimes_{\field{C}}g^*(\mathcal{L}_{B_1}) ) \otimes_{\phi} C^*(\Gamma_2).
\end{eqnarray*}
The isomorphism $\alpha_0$ is the isomorphism from Proposition \ref{MisBunProp}.
\end{define}

\begin{prop}
The map $\mu_{\phi}$ is well-defined.
\end{prop}

\begin{proof}
We must show that the map respects each of the three relations. The case of the disjoint union/direct sum relation is trivial. 

For the bordism relation, suppose that $((Z,W),F,(h_2,h_1))$ is a cycle with boundary which has boundary $(W,E,(f,g))$. Then 
\begin{eqnarray*}
 \mu_{\phi}(W,E,(f,g)) & = & (W, (E \otimes_{\field{C}}f^*(\mathcal{L}_{B\Gamma_2}) ,E|_{\partial W} \otimes_{\field{C}}g^*(\mathcal{L}_{B\Gamma_1}) ,\alpha)) \\
& = & \partial ((Z, W), (F\otimes_{\field{C}}h_2^*(\mathcal{L}_{B\Gamma_2}), F|_{\partial Z - {\rm int}(W)}\otimes_{\field{C}}h_1^*(\mathcal{L}_{B\Gamma_1}), \tilde{\alpha})) 
\end{eqnarray*} 
where $\tilde{\alpha}$ is constructed in the same way as $\alpha$. Hence $\mu_{\phi}(W,E,(f,g))$ is a boundary whenever $(W,E,(f,g))$ is a boundary. 

The case of vector bundle modification follows by observing that, if $V$ is a spin$^c$-vector bundle with even dimensional fibers over $W$, then $\mu_{\phi}( (W,E,(f,g))^V) = (\mu_{\phi}(W,E,(f,g)))^V$.
\end{proof}

\begin{theorem}
\label{sesforbphi} 
(compare with Theorem 2.17 in \cite{CWY}) \\
Let $\phi:\Gamma_1\to \Gamma_2$ be a group homomorphism and $h:Y\to X$, $f:X\to B\Gamma_2$ and $g:Y\to B\Gamma_1$ be continuous mappings such that $f\circ h=g\circ B\phi$. The mapping $\mu_{\phi,h}:=\mu_\phi\circ (f,g)_*$ fits into a commuting diagram with exact rows:
\small
\begin{center}
 \minCDarrowwidth15pt
$\begin{CD}
@>>> K_*^{\textnormal{geo}}(Y) @>h_*>> K_*^{\textnormal{geo}}(X) @>r>> K_*^{\textnormal{geo}}(h) @>\delta >> K_{*+1}^{\textnormal{geo}}(Y) @>>> \\
@. @VV \mu_{Y} V @VV\mu_{X}V  @VV\mu_{\phi,h} V  @VV\mu_{Y} V \\
@>>> K_*^{\textnormal{geo}}(pt;C^*(\Gamma_1)) @>\phi >> K_*^{\textnormal{geo}}(pt;C^*(\Gamma_2)) @>r>> K_*^{\textnormal{geo}}(pt;\phi) @>\delta>> K_{*+1}^{\textnormal{geo}}(pt;C^*(\Gamma_1)) @>>> \\
\end{CD}$
\end{center}
\normalsize
where 
\begin{enumerate}
\item $\mu_X :K_*(X) \rightarrow K_*(pt;C^*(\Gamma_2))$ is the composition of $f_*$ with the free assembly map as defined at the level of geometric cycles in the introduction (see Equation \eqref{geometricassemblyequation} on page \pageref{geometricassemblyequation}), $\mu_Y$ is defined similarly;
\item the top long exact sequence is from the statement of Theorem \ref{relGeoKhomDomExaSeq};
\item the bottom long exact sequence is from Theorem \ref{exactonk}.
\end{enumerate}
In particular, if the map $\mu$ is an isomorphism for both $\Gamma_1$ and $\Gamma_2$, then $\mu_{\phi}$ is also an isomorphism.
\end{theorem}

\begin{proof}
The proof follows directly from the definitions of the maps: they are each defined at the level of cycles and the diagram actually commutes at the level of cycles. For example, if $(M, E, f)$ is a cycle in $K_*^{\textnormal{geo}}(B\Gamma_2)$, then 
$$ (\mu_{\phi} \circ r)(M, E, f) = (M, (E\otimes f^*(\mathcal{L}_{B\Gamma_2}), \emptyset))= (r \circ \mu_{\Gamma_2})(M,E,f)$$
The remaining details of the proof are omitted.
\end{proof}

\section{The isomorphism between the geometric and analytic realizations of the mapping cone} 
\label{IsoGeoToAna}

Our first goal is the construction of a suitable isomorphism $K^{\textnormal{geo}}_*(pt;\phi)\cong \linebreak[4] KK^{*}(\field{C}, SC_{\phi})$, where $B_1$ and $B_2$ are unital $C^*$-algebras and $\phi: B_1 \rightarrow B_2$ is a unital $*$-homomorphism. In fact, in \cite{DeeMappingCone}, an abstract isomorphism between $K^{\textnormal{geo}}_*(X;\phi)$ and $KK^{*}(C(X), SC_{\phi})$ for $X$ a finite $CW$-complex is constructed. The abstract isomorphism from \cite{DeeMappingCone} is defined from a type of relative topological index and is difficult to use in analytic applications. The isomorphism we aim at constructing extends an analytic construction from \cite{DeeMappingCone}: under the assumption that the induced map $\phi_*: K_*(B_1) \rightarrow K_*(B_2)$ is injective, it was shown in \cite{DeeMappingCone} that the higher Atiyah-Patodi-Singer index induces an isomorphism from $K^{\textnormal{geo}}_*(pt;\phi)$ to $KK^*(\field{C}, SC_{\phi})$. The reader is directed to \cite[Theorems 4.6 and 4.7]{DeeMappingCone} for the precise statement. However, in general, the higher Atiyah-Patodi-Singer index does {\it not} induce an isomorphism from $K^{\textnormal{geo}}_*(pt;\phi)$ to $KK^{*}(\field{C}, SC_{\phi})$, it does not even produce a well defined mapping \footnote{ Indeed, the image in $KK^{*}(\field{C}, SC_{\phi})$ of the Atiyah-Patodi-Singer index of a geometric cycle depends on the choice of trivializing operator if $\ker \phi_*\neq 0$. Additionally, the group $KK^{*}(\field{C}, SC_{\phi})$ is not exhausted by Atiyah-Patodi-Singer indices if $\ker\phi_*\neq 0$ because $x\in KK^{*}(\field{C}, SC_{\phi})$ is an Atiyah-Patodi-Singer index if and only if $\delta(x)=0$, i.e. $x\in \mathrm{im}\, r$. Thus $\mathrm{im} \,\delta=\ker \phi_*$ obstructs the Atiyah-Patodi-Singer indices exhausting $KK^{*}(\field{C}, SC_{\phi})$.}. 
\par

\subsection{Analytic realization of cycles on mapping cone} 
Our goal is a generalization of the construction in \cite{DeeMappingCone} to an \emph{analytically} defined isomorphism between $K^{\textnormal{geo}}_*(pt;\phi)$ and $KK^{*}(\field{C},SC_{\phi})$ with no assumptions on $\phi_*$. As mentioned above, this result will be used to relate the geometrically defined assembly map of Section \ref{GeoMap} to the work of Chang, Weinberger, and Yu (i.e., the map, $\mu_{CWY}$, discussed in Section \ref{CWYmap}). The general idea behind the isomorphism $K^{\textnormal{geo}}_*(pt;\phi)\cong KK^{*}(\field{C}, SC_{\phi})$ is to map a cycle for $K^{\textnormal{geo}}_*(pt;\phi)$ (equipped with additional geometric data) to a class in $KK^*(\field{C},B_2)$ using higher Atiyah-Patodi-Singer index theory. The associated class in $KK^*(\field{C},SC_{\phi})$ in general depends heavily on these choices. To correct for the choices made, we construct an additional term in $KK^{*}(\field{C},SC_{\phi})$ which also depends heavily on said choices, it is defined purely using data on boundary of the cycle. In particular, in the rather special case when the boundary of the cycle in $K^{\textnormal{geo}}_*(pt;\phi)$ is empty, this construction is compatible with the isomorphism from $ K_*^{\textnormal{geo}}(pt;B_2)$ to $K_*(B_2)$ defined via higher index theory.
\par

Although we are mostly interested in the case of a $*$-homomorphism induced $\phi: \Gamma_1 \rightarrow \Gamma_2$, the results of this section are more general. As such, let $\phi: B_1 \rightarrow B_2$ be a unital $*$-homomorphism between unital $C^*$-algebras. For a Riemannian spin$^c$-manifold with boundary $W$, we let $S_W\to W$ denote the associated Clifford bundle of complex spinors.

\begin{define}
\label{diracsandtrivi}
A choice of Dirac operators on a cycle $(W,(E_{B_2},F_{B_1},\alpha))$ for $K^{\textnormal{geo}}_*(pt;\phi)$ is a pair $(D_E^W,D_F^{\partial W})$ of a spin$^c$-Dirac operator $D_F^{\partial W}$ on $F_{B_1}\to \partial W$ and a spin$^c$-Dirac operator $D_E^W$ on $E_{B_2}\to W$ being of product type near $\partial W$ with boundary operator $D_E^{\partial W}=\alpha^*(D_F^{\partial W}\otimes_\phi B_2)$.

A trivializing operator for a cycle $(W,(E_{B_2},F_{B_1},\alpha),(D_E^W,D_F^{\partial W}))$ with Dirac operators is a self-adjoint $A\in \Psi^{-\infty}_{B_2}(\partial W; S_{\partial W}\otimes E_{B_2}|_{\partial W})$  such that $D_E^{\partial W}+A$ is invertible. We also refer to $A$ as a trivializing operator for $D_E^{\partial W}$.
\end{define}

By \cite[Theorem 3]{LP} and \cite[Proposition 10]{LPGAFA} any cycle with Dirac operators admits a trivializing operator. The reason for introducing trivializing operators is that they are unavoidable when doing higher Atiyah-Patodi-Singer index theory because the Dirac operator on the boundary is in general not invertible. We first construct a class in $K_1(C_\phi)$ from an even-dimensional cycle with Dirac operators and trivializing operator. We then prove that the constructed class in $K_1(C_\phi)$ is independent of involved choices. The construction for odd-dimensional cycles follows by a formal suspension, see \cite[Section 3.2]{LP} or \cite[Section 2]{zeid14}. 

Take $\chi,\tilde{\chi}\in C^\infty((0,1],\mathbb{R})$ such that $\chi\geq 1$ and 
\begin{equation}
\label{chichitilde}
\chi(t)=\begin{cases} t^{-1}, \; &\mbox{near $t=0$}\\ 1, \; &\mbox{near $t=1$}\end{cases}\qquad \mbox{and} \qquad \tilde{\chi}(t)=\begin{cases} t^{-1}, \; &\mbox{near $t=0$}\\ 0, \; &\mbox{near $t=1$}\end{cases}.
\end{equation}
The set of such pairs $(\chi,\tilde{\chi})$ form a convex subset of $C^\infty(0,1]\times C^\infty(0,1]$ and is therefore path connected. Define the following families of operators on $S_{\partial W}\otimes E|_{\partial W}\to \partial W$: 
\begin{equation}
\label{pathofops}
\hat{A}_t:=\tilde{\chi}(t)A\quad\mbox{and}\quad \hat{D}_{E,t}^\partial:=\chi(t)D^{\partial W}_E,\quad \mbox{for}\;\; t\in (0,1].
\end{equation}
For any $t\in (0,1]$, $\hat{D}_{E,t}^\partial+\hat{A}_t$ is a self-adjoint elliptic element in the Fomenko-Mishchenko calculus $\Psi^{*}_{B_2}(\partial W; S_{\partial W}\otimes E_{B_2}|_{\partial W})$. Hence it is a self-adjoint regular operator with compact resolvent on the $B_2$-Hilbert module $L^2(\partial W;S_{\partial W}\otimes E_{B_2})$. Let $c$ denote the Cayley transform:
\begin{equation}
\label{cayleydeef}
c(x):=\frac{ix+1}{ix-1}=1+2(ix-1)^{-1}.
\end{equation}

\begin{remark}
If $(W,(E_{B_2},F_{B_1},\alpha))$ is a cycle, $\alpha$ induces a $*$-homomorphism $\tilde{\alpha}:\mathbb{K}_{B_1}(L^2(\partial W,S_{\partial W}\otimes F_{B_1}))\to \mathbbm{K}_{B_2}(L^2(\partial W;S_{\partial W}\otimes E_{B_2}))$. 
\end{remark}

\begin{prop}
\label{firstunit}
The operator-valued function $c(\hat{D}^\partial_E+\hat{A})$ given by
$$(0,1]\ni t\mapsto c(\hat{D}^\partial_{E,t}+\hat{A}_t)\in 1+\mathbbm{K}_{B_2}(L^2(\partial W;S_{\partial W}\otimes E_{B_2})),$$
defines an element in $1+C_0((0,1],\mathbbm{K}_{B_2}(L^2(\partial W;S_{\partial W}\otimes E_{B_2})))$. Moreover, 
$$c(\hat{D}^\partial_E+\hat{A})(1)=\tilde{\alpha}(c(D_F)).$$
\end{prop}

\begin{proof}
The formula \eqref{cayleydeef} and the compact resolvent of $\hat{D}^\partial_{E,t}+\hat{A}_t$ for $t\in (0,1]$ imply that $c(\hat{D}^\partial_E+\hat{A})\in C((0,1],1+\mathbbm{K}_{B_2}(L^2(\partial W;S_{\partial W}\otimes E_{B_2})))$. For $|x|\geq 1$, we have the asymptotic formula $c(t^{-1}x)=1+O(t)$ uniformly in $x$ as $t\to 0$. Therefore, functional calculus for self-adjoint regular operators implies that $c(\hat{D}^\partial_{E,t}+\hat{A}_t)\to 1$ in norm as $t\to 0$ and $c(\hat{D}^\partial_E+\hat{A})\in 1+C_0((0,1],\mathbbm{K}_{B_2}(L^2(\partial W;S_{\partial W}\otimes E_{B_2}))$
\end{proof}

Recall that an $(A,B)$-Hilbert $C^*$-module $E_B$ is called a correspondence if $A$ acts as $B$-compact operators on $E_B$. A correspondence $E_B$ gives rise to a class $[(E_B,0)]\in KK_0(A,B)$ and a mapping $K_*(B)\to K_*(A)$.

\begin{remark}
The $C_\phi$-Hilbert $C^*$-module 
$$\mathcal{M}_\alpha:=\{\xi\oplus \eta\in C_0((0,1],L^2(\partial W;S_{\partial W}\otimes E_{B_2}))\oplus L^2(\partial W,S_{\partial W}\otimes F_{B_1}):\; \xi(1)=\alpha^*\eta\},$$
induces a correspondence from $C_{\tilde{\alpha}}$ to $C_\phi$. If $E_{B_2}$ and $F_{B_1}$ are full, then $\mathcal{M}_\alpha$ is a Morita equivalence. We tacitly identify elements in $K_*(C_{\tilde{\alpha}})$ with their image in $K_*(C_\phi)$ under this correspondence.
\end{remark}

\begin{define}
We define the Cayley transform of $(W,(E_{B_2},F_{B_1},\alpha),(D_E^W,D_F^{\partial W}),A)$ as the class 
$$c(W,(E_{B_2},F_{B_1},\alpha),(D_E^W,D_F^{\partial W}),A):=[c(\hat{D}^\partial_E+\hat{A})\oplus c(D_F)]\in K_1(C_\phi).$$
\end{define} 

Since $\chi$ and $\tilde{\chi}$ are chosen from path connected spaces, we have the following:

\begin{prop}
The class $c(W,(E_{B_2},F_{B_1},\alpha),(D_E^W,D_F^{\partial W}),A)$ does not depend on the choice of smooth functions $\chi$ and $\tilde{\chi}$ satisfying \eqref{chichitilde}.
\end{prop}

For an odd operator $T$ on a graded Hilbert $C^*$-module $\pmb{E}=\pmb{E}^+\oplus \pmb{E}^-$, we write $T^+$ for the induced operator $\pmb{E}^+\to \pmb{E}^-$.

\begin{define}
\label{notationaps}
Let $B$ be a $C^*$-algebra and $W$ a compact manifold with boundary. Assume that $D^W_{E_B}$ is a Dirac operator twisted by a $B$-bundle $E_B\to W$ of product type near $\partial W$ and that $A\in \Psi^{-\infty}(\partial W, S_{\partial W} \otimes E_B|_{\partial W})$ is a trivializing operator for the boundary operator $D^{\partial W}_{E_B}$ (see Definition \ref{diracsandtrivi}). The APS-realization $D^W_{E_B}(A)$ is a regular self-adjoint odd $B$-Fredholm operator densely defined on $L^2(W,E_B\otimes S_W)$ defined by declaring $D^W_{E_B}(A)^+$ to be the restriction of the densely defined operator $(D^W_{E_B})^+$ between the $B$-Hilbert $C^*$-modules $L^2(W,E_B\otimes S_W^+)\to L^2(W,E_B\otimes S_W^-)$ to the domain defined from APS-boundary conditions
$$\mathrm{Dom}(D^W_{E_B}(A)):=\{\xi \in H^1(W,E_B\otimes S_W^+): \chi_{[0,\infty)}(D^{\partial W}_{E_B}+A)\xi|_{\partial W}=0\},$$
where we identify $S_W^+|_{\partial W}=S_{\partial W}$. We define ${\rm ind}_{APS}(D^W_{E_B},A):=\mathrm{ind}_B(D^W_{E_B}(A)^+)\in K_*(B)$ as the index of the APS-realization of $D^W_{E_B}$ associated with $A$. 
\end{define}

For details regarding higher APS-theory, see \cite[Section 3]{LP}. The maximal domain of the Dirac operator $D^W_{E_B}$ on $L^2(W,E_B\otimes S_W)$ is poorly behaved; it is the APS-boundary conditions that make it Fredholm. However, without additional assumptions on the Dirac operator or a trivializing operator it is not well defined for a general $C^*$-algebra $B$. Let $j_*:K_0(B_2)\cong K_1(C_0(0,1)\otimes B_2)\to K_1(C_\phi)$ denote the composition of the Bott mapping and the mapping induced from the inclusion $i:C_0(0,1)\otimes B_2\to C_\phi$. 

\begin{lemma}
\label{cyctoclasses}
Let $(W,(E_{B_2},F_{B_1},\alpha),(D_E^W,D_F^{\partial W}))$ be a cycle with Dirac operators and $A$ and $A'$ two trivializing operators. Then 
\begin{align*}
c(W,(E_{B_2},F_{B_1},\alpha),&(D_E^W,D_F^{\partial W}),A)\\
-c(W,(E_{B_2},&F_{B_1},\alpha),(D_E^W,D_F^{\partial W}),A')\\
=&j_*\left(\mathrm{ind}_{APS}(D_E^W,A')-\mathrm{ind}_{APS}(D^W_E,A)\right).
\end{align*}
In particular, the class 
\begin{align}
\nonumber
\Phi_{{\rm cone}}(W,(E_{B_2},&F_{B_1},\alpha))\\
\label{definofphicone}&:=c(W,(E_{B_2},F_{B_1},\alpha),(D_E^W,D_F^{\partial W}),A)+j_*\mathrm{ind}_{APS}(D_E^W,A),
\end{align}
is independent of choice of Dirac operators and trivializing operator.
\end{lemma}

\begin{proof}
The second statement of the lemma follows from the first: the choice of Dirac operators is from a path-connected space and by making a choice of a continuous path of trivializing operators, the class $\Phi_{{\rm cone}}(W,(E_{B_2},F_{B_1},\alpha))$ is well defined due to homotopy invariance of $K$-theory. Such an argument is standard, see for instance \cite[Lemma 3.4]{DGIII} and \cite[Section 2.5]{LP}.

To prove the first statement, we use the following elementary fact in $K$-theory: suppose that $u$ and $\tilde{u}$ are unitaries in the unitalization of $C_0((0,1], B_2\otimes \mathbbm{K})$ such that $u(1)=\tilde{u}(1)$. Then, for any unitary $U$ in the unitalization of $B_1\otimes \mathbbm{K}$ with $u(1)=\phi(U)$,
$$[u\oplus U]-[\tilde{u}\oplus U]=[u\#\tilde{u}_{op}] \in \mathrm{im}(K_1(C_0(0,1)\otimes B_2)\to K_1(C_\phi)),$$
where $\tilde{u}_{op}(t)=\tilde{u}(1-t)$ and $\#$ denotes concatenation of paths defined by 
$$u\#\tilde{u}_{op}(t)=\begin{cases}
u(2t), \;& t\in [0,\frac{1}{2}),\\
\tilde{u}_{op}(2t-1), \;& t\in [\frac{1}{2},1].
\end{cases}$$
In our case, this fact implies that
\begin{align*}
c(W,(E_{B_2},F_{B_1},\alpha),&(D_E^W,D_F^{\partial W}),A)-c(W,(E_{B_2},F_{B_1},\alpha),(D_E^W,D_F^{\partial W}),A')\\
&=i_*[c(\hat{D}^\partial_E\#\hat{D}^\partial_{E,op}+\hat{A}\#\hat{A}'_{op})],
\end{align*}
where $\hat{D}^\partial_E\#\hat{D}^\partial_{E,op}$ is of the form $g(t)D^{\partial}_E$, for a function $g$ behaving like $t^{-1}$ near $0$ and $(1-t)^{-1}$ near $1$, and $\hat{A}\#\hat{A}'_{op}$ is the concatenation of the path $\hat{A}$ with the path $\hat{A}'_{op}(t)=\hat{A}'(1-t)$. We write $\check{D}$ for the path $\hat{D}^\partial_E\#\hat{D}^\partial_{E,op}+\hat{A}\#\hat{A}'_{op}$, this path coincides with $t^{-1}(D^\partial_E+A)$ near $t=0$ and with $(1-t)^{-1}(D^\partial_E+A')$ near $t=1$.

Consider the choice function $f(x):=\frac{1}{\pi}\tan^{-1}(x)+\frac{1}{2}$. The function $f\in C^\infty(\field{R})$ satisfies $f(+\infty)=1$, $f(-\infty)=0$ and $c(x)=\mathrm{e}^{2\pi if(x)}$. We define 
$$P_A:=\chi_{[0,\infty)}(D^\partial_E+A)\equiv f(\check{D})(0)\quad\mbox{and}\quad P_{A'}:=\chi_{[0,\infty)}(D^\partial_E+A')\equiv f(\check{D})(1).$$ 
We also pick a spectral section $Q_0$ for $D_E^\partial$ such that $Q_0-P_A$ and $Q_0-P_{A'}$ are projections in $\mathbbm{K}_{B_2}(L^2(\partial W;S_{\partial W}\otimes E_{B_2}))$ (see \cite[Corollary 1, p. 366]{LP}). Consider the continuous path
$$p(t)=\begin{cases}
(1-2t)P_A+2tQ_0, \;& t\in [0,\frac{1}{2}),\\
(2t-1)P_{A'}+2(1-t)Q_0, \;& t\in [\frac{1}{2},1].
\end{cases}$$
In fact, $p$ is a constant projection modulo $C[0,1]\otimes \mathbbm{K}_{B_2}(L^2(\partial W;S_{\partial W}\otimes E_{B_2}))$ satisfying $p(0)=P_A$ and $p(1)=P_{A'}$ in the end-points. Since $f(\check{D})-p\in C_0(0,1)\otimes \mathbbm{K}_{B_2}(L^2(\partial W;S_{\partial W}\otimes E_{B_2}))$ we conclude the identity
$$[c(\check{D})]=[\mathrm{e}^{2\pi if(\check{D})}]=[\mathrm{e}^{2\pi ip}],$$
as classes in $K_1(C_0(0,1)\otimes B_2)$.

On the other hand, \cite[Theorem 6]{LP}, implies that 
$$\mathrm{ind}_{APS}(D_E^W,A)-\mathrm{ind}_{APS}(D^W_E,A')=[P_{A'}-P_A]\equiv [Q_0-P_A]-[Q_0-P_{A'}].$$
Recall that $Q_0$ is choosen such that $Q_0-P_A$ and $Q_0-P_{A'}$ are projections in $\mathbbm{K}_{B_2}(L^2(\partial W;S_{\partial W}\otimes E_{B_2}))$. By definition, we have the identity 
$$j_*[P_{A'}-P_A]=[\mathrm{e}^{2\pi i t(Q_0-P_A)}]-[\mathrm{e}^{2\pi i t(Q_0-P_{A'})}]=[\mathrm{e}^{2\pi i t(Q_0-P_A)}\#\mathrm{e}^{2\pi i t(P_{A'}-Q_0)}].$$
We write $2\pi i q$ for the path appearing in the exponent on the right hand side, i.e.
$$q(t)=\begin{cases}
2t(Q_0-P_A), \;& t\in [0,\frac{1}{2}),\\
(2t-1)(P_{A'}-Q_0), \;& t\in [\frac{1}{2},1].
\end{cases}$$
The path $q$ is not continuous, but $\mathrm{e}^{2\pi i q}$ is. Since 
$$p(t)-q(t)=\begin{cases}
P_A, \;& t\in [0,\frac{1}{2}),\\
Q_0, \;& t\in [\frac{1}{2},1].
\end{cases}$$ 
is a projection that commutes with $p(t)$ and $q(t)$ for each $t$, we have the identity $\mathrm{e}^{2\pi i q(t)}=\mathrm{e}^{2\pi i p(t)}$ for each $t$. In summary, $[c(\check{D})]=[\mathrm{e}^{2\pi ip}]=[\mathrm{e}^{2\pi iq}]=j_*[P_{A'}-P_A]$ proving the lemma. 

\end{proof}

\begin{theorem} 
\label{defIndMap}
The construction of $\Phi_{{\rm cone}}$ in Equation \eqref{definofphicone} produces a well defined isomorphism 
$$\Phi_{{\rm cone}}:K_*^{\textnormal{geo}}(pt;\phi)\to K_{*+1}(C_\phi),$$
fitting into a commutative diagram with exact rows and horizontal mappings being isomorphisms:
\begin{equation}
\label{phiconemap}
 \minCDarrowwidth15pt
\begin{CD}
\cdots@>\phi_*>> K_*^{\textnormal{geo}}(pt;B_2)) @>r>> K_*^{\textnormal{geo}}(pt;\phi) @>\delta >> K_{*+1}^{\textnormal{geo}}(pt;B_1) @>\phi_*>>\cdots \\
@.  @VVV  @VV\Phi_{\textnormal{cone}} V  @VV V \\
\cdots@>\phi_*>> K_*(B_2)) @>r>> K_{*+1}(C_\phi) @>\delta >> K_{*+1}(B_1) @>\phi_*>> \cdots\\
\end{CD},
\end{equation}
\end{theorem}

\begin{proof}
Assuming that $\Phi_{\textnormal{cone}}$ is well defined, it follows from its construction that the diagram \eqref{phiconemap} commutes and that it is an isomorphism using the five lemma. The proof that $\Phi_{\textnormal{cone}}$ is well defined is divided into proving that the map respects the bordism and vector bundle modification relations; that it respect the disjoint union/direct sum relation follows from basic facts in higher APS-index theory. 

Consider a cycle $(W_0,(\tilde{F}_{B_1}\otimes _\phi B_2, \tilde{F}_{B_1}|_{\partial W_0},u|_{\partial W_0}))$ defined from a $B_1$-bundle $\tilde{F}_{B_1}\to W_0$ equipped with an automorphism $u:\tilde{F}\otimes_\phi B_2\to \tilde{F}\otimes_\phi B_2$. We can choose a Dirac operator $D_{\tilde{F}}^{W_0}$ on $\tilde{F}_{B_1}\to W_0$ and a trivializing operator $A_1\in  \Psi^{-\infty}_{B_1}(\partial W_0; S_{\partial W_0}\otimes \tilde{F}_{B_1}|_{\partial W_0})$ for $D_{\tilde{F}}^{\partial}$. By construction, $(D_{\tilde{F}}^{W_0}\otimes _\phi 1_{B_2}, D_{\tilde{F}}^\partial)$ is a Dirac operator and $A_1\otimes_\phi 1_{B_2}$ a trivializing operator for the cycle with Dirac operator $(W_0,(\tilde{F}_{B_1}\otimes _\phi B_2, \tilde{F}_{B_1}|_{\partial W_0},\mathrm{id}_{\tilde{F}_{B_1}|_{\partial W_0}\otimes _\phi B_2}),(D_{\tilde{F}}^{W_0}\otimes _\phi 1_{B_2}, D_{\tilde{F}}^\partial))$. It holds that 
\begin{align}
\nonumber
j_*\mathrm{ind}_{APS}(D_{\tilde{F}}^{W_0}\otimes _\phi 1&_{B_2},A_1\otimes _\phi 1_{B_2})=j_*\circ \phi_*\mathrm{ind}_{APS}(D_{\tilde{F}}^{W_0},A_1)=0,\quad\mbox{and}\\
\nonumber
c(\hat{D}^\partial_{\tilde{F}}\otimes_\phi 1_{B_2}+\hat{A}_1&\otimes_\phi 1_{B_2})\oplus c(D_{\tilde{F}}^\partial)\\
\label{somehom}
&\sim_h c(\chi\!\cdot\!(D^\partial_{\tilde{F}}\otimes_\phi 1_{B_2}+A_1\otimes_\phi 1_{B_2}))\oplus c(D_{\tilde{F}}^\partial+A_1)\sim_h 1\oplus 1.
\end{align}
The homotopies in Equation \eqref{somehom} are inside the group of unitaries in the unitalization of the mapping cone of the mapping 
$$\mathbbm{K}_{B_1}(L^2(\partial W;S_{\partial W}\otimes F_{B_1}))\to \mathbbm{K}_{B_2}(L^2(\partial W;S_{\partial W}\otimes E_{B_2}))$$
constructed from $\alpha$. We remark that an immediate consequence of the first homotopy in Equation \eqref{somehom} is that the Cayley transform term is in the image of $K_1(C_0(0,1]\otimes C^*(\Gamma_1))\to K_1(C_\phi)$. In conclusion, 
$$\Phi_{\rm cone}(W_0,(\tilde{F}_{B_1}\otimes _\phi B_2, \tilde{F}_{B_1}|_{\partial W_0},u|_{\partial W_0}))=0.$$

We can now prove bordism invariance. Let $(W,(E_{B_2},F_{B_1},\alpha))$ be a nullbordant cycle. By assumption, there is a spin$^c$-manifold with boundary $Z$ containing $W$ as a regular domain, with $E_{B_2}$ extending to a bundle $\tilde{E}_{B_2}\to Z$, $F_{B_1}$ extending to a bundle $\tilde{F}_{B_1}\to W_0:=\partial Z\setminus W^\circ$ and $\alpha$ to an isomorphism $u:\tilde{E}_{B_2}|_{W_0}\to \tilde{F}_{B_1}\otimes_\phi B_2$. By the argument above, $\Phi_{\rm cone}(W_0,(\tilde{E}_{B_2}|_{W_0},\tilde{F}_{B_1},u))=0$. Additivity of the index and of the Cayley transform shows that 
\begin{align*}
\Phi_{\rm cone}&(W,(E_{B_2},F_{B_1},\alpha))\\
&=\Phi_{\rm cone}(W,(E_{B_2},F_{B_1},\alpha))+\Phi_{\rm cone}(W_0,(\tilde{E}_{B_2}|_{W_0},\tilde{F}_{B_1},u))\\
&=\Phi_{\rm cone}(\partial Z,(\tilde{E}_{B_2},\emptyset,\emptyset))=0,
\end{align*}
where the last identity follows from the bordism invariance of the index.

To show that $\Phi_{\rm cone}$ respects vector bundle modification, we use of a trick from \cite{DGIII} (also see \cite{DeeMappingCone,DGMbor}). Let $(W,(E_{B_2},F_{B_1},\alpha), (D_E,D_F),A)$ be a cycle with Dirac operators and trivializing operator and $V\to W$ a spin$^c$-vector bundle of even rank. The vector bundle modification of $(W,(E_{B_2},F_{B_1},\alpha))$ is given by $(W^V,(E_{B_2}^V,F_{B_1}^V,\alpha^V))$ where $W^V:=S(V\oplus 1_\field{R})$ (as in Definition \ref{vectormod}) and 
$$(E_{B_2}^V,F_{B_1}^V,\alpha^V):=(p^*E_{B_2}\otimes B,p^*F_{B_1}\otimes B,p^*\alpha\otimes \mathrm{id}_B),$$ 
where $p:W^V\to W$ is the projection and $B\to W^V$ denotes the Bott bundle. Following \cite[Section 2.3]{DGIII}, we can vector bundle modify Dirac operators and trivializing operators. We let $(D_E^V,D_F^V,A^V)$ denote the vector bundle modification of $(D_E^W,D_F^{\partial W},A)$ (see \cite[Definition 2.11 and 2.18]{DGIII}). The vector bundle modification of a cycle with Dirac operators and trivializing operators is then given by
$$(W,(E_{B_2},F_{B_1},\alpha), (D_E^W,D_F^{\partial W}),A)^V:=(W^V,(E_{B_2}^V,F_{B_1}^V,\alpha^V), (D_E^V,D_F^V),A^V).$$
By \cite[Section 2.3]{DGIII}, $\mathrm{ind}_{APS}(D_E,A)=\mathrm{ind}_{APS}(D_E^V,A^V)$. Furthermore, again using \cite[Section 2.3]{DGIII}, we can decompose the path of operators constructed in \eqref{pathofops} as
$$\hat{D}_E^V+\hat{A}^V=(\hat{D}_E+\hat{A})\oplus \hat{D}_E^\perp\quad \mbox{and}\quad D_F^V=D_F\oplus D_F^\perp,$$
where $D_F^\perp$ and $\hat{D}_E^\perp$ are invertible self-adjoint regular operators, $\hat{D}_E^\perp$ being a continuous family over $(0,1]$ coinciding with $\alpha^*(D_F^\perp\otimes_\phi 1_{B_2})$ at $t=1$. Therefore, $c(\hat{D}_E^\perp)\oplus c(D_F^\perp)$ defines a trivial class in $K$-theory whose vanishing is implemented by the homotopy $(c(s^{-1}\hat{D}_E^\perp)\oplus c(s^{-1}D_F^\perp))_{s\in [0,1]}$. It follows that 
$$c(W,(E_{B_2},F_{B_1},\alpha), (D^W_E,D^{\partial W}_F),A)=c(W^V,(E_{B_2}^V,F_{B_1}^V,\alpha^V),(D^V_E,D^{V}_F),A^V).$$
\end{proof}

We introduce some further notation. Let $B$ denote a $C^*$-algebra. For a Dirac operator $D^M_{E_B}$ twisted by a $B$-bundle $E_B\to M$ on the closed manifold $M$ we write ${\rm ind}_{AS}(D^M_{E_B})\in K_{*}(B)$ for its index. The index is well defined because $D^M_{E_B}$ is elliptic in the Mischenko-Fomenko calculus. 

We return to the special case where $B_1=C^*(\Gamma_1)$, $B_2=C^*(\Gamma_2)$, and $\phi$ is induced from a group homomorphism.

\begin{lemma}
\label{thelemmaformerlyknownasproposition3.12}
Let $(W,E_{\field{C}},(f,g))$ be a cycle in $K_*^{\textnormal{geo}}(B\phi)$ satisfying 
$${\rm ind}_{AS}(D^{\partial W}_{E_{\field{C}}\otimes g^*(\mathcal{L}_{B\Gamma_1})})=0\in K_{*-1}(C^*(\Gamma_1)),$$ 
for a Dirac operator $D^{\partial W}_{E_{\field{C}}|_{\partial W}}$ twisted by the flat connection on $g^*(\mathcal{L}_{B\Gamma_1})$. Then, 
\begin{equation}
\label{muequalr}
\mu(W,E,(f,g))=\Phi_{\rm cone}^{-1}(j_*({\rm ind}_{APS}(D^{ W}_{E_{\field{C}} \otimes f^*(\mathcal{L}_{B\Gamma_2})},\alpha_*\phi_*A')))
\end{equation}
where
\begin{enumerate}
\item $D^{ W}_{E_{\field{C}|_{\partial W}} \otimes f^*(\mathcal{L}_{B\Gamma_2})}$ is constructed from any Dirac operator $D^{W}_{E_{\field{C}}}$ on $W$, being of product type near $\partial W$ with boundary operator $D^{\partial W}_{E_{\field{C}}}$, by twisting it with the flat connection on $f^*(\mathcal{L}_{B\Gamma_2})$.
\item ${\rm ind}_{APS}( \: \cdot \:)$ denotes the higher Atiyah-Patodi-Singer index from Definition \ref{notationaps} , again see \cite{LP} and references therein; it is an element in $K_*(C^*(\Gamma_2))$;
\item $A'\in \Psi^{-\infty}_{C^*(\Gamma_1)}(\partial W; S_{\partial W}\otimes E_{\field{C}}|_{\partial W} \otimes g^*(\mathcal{L}_{B\Gamma_1}))$ is a trivializing operator (see \cite[Section 2]{LP}) for $D_{\partial W, E_{\field{C}}\otimes g^*(\mathcal{L}_{B\Gamma_1})}$;
\item The isomorphism $\alpha$ is constructed as in Definition \ref{defofassembly} (on page \pageref{defofassembly}).
\end{enumerate}
\end{lemma}

\begin{proof}
By definition, $\Phi_{\textnormal{cone}}( \mu((W,E,(f,g))))$ is equal to
\begin{align}
\nonumber
c(W, \left( E_{\field{C}}\otimes f^*(\mathcal{L}_{B \Gamma_2}), E_{\field{C}}|_{\partial W} \otimes g^*(\mathcal{L}_{B \Gamma_1}), \alpha \right), & (D_{E_{\field{C}}\otimes f^*(\mathcal{L}_{B \Gamma_2})}^W, D_{E_{\field{C}}|_{\partial W} \otimes g^*(\mathcal{L}_{B \Gamma_1})}^{\partial W}), A ) \\
\label{phiconeandindaps}
& +j_*({\rm ind_{APS}}(D_{E_{\field{C}}\otimes f^*(\mathcal{L}_{B\Gamma_2}) }^W, A)) 
\end{align}
where, by Proposition \ref{cyctoclasses}, the resulting class is independent of the choice of trivializing operator $A$. 

The assumption that ${\rm ind}_{AS}(D^{\partial W}_{E_{\field{C}}|_{\partial W}\otimes g^*(\mathcal{L}_{B\Gamma_1})})=0$ implies that we can take $A$ of the form $\alpha_*\phi_*(A')$ where $A'$ is a trivializing operator for the operator $D^{\partial W}_{E_{\field{C}}|_{\partial W}\otimes g^*(\mathcal{L}_{B\Gamma_1})}$. We note that $A'$ is an element of $\Psi^{-\infty}_{C^*(\Gamma_1)}(\partial W; S_{\partial W}\otimes E_{\field{C}}|_{\partial W} \otimes g^*(\mathcal{L}_{B\Gamma_1}))$ and results in \cite[Appendix C]{PS} imply that $\alpha_*\phi_*(A')$ is a trivializing operator for $D_{E_{\field{C}}\otimes f^*(\mathcal{L}_{B \Gamma_2})}^{ \partial W}$. Using similar homotopies to those in Equation \eqref{somehom}, it follows that 
\small
$$c(W, \left( E_{\field{C}}\otimes f^*(\mathcal{L}_{B \Gamma_2}), E_{\field{C}}|_{\partial W} \otimes g^*(\mathcal{L}_{B \Gamma_1}), \alpha \right),  (D_{E_{\field{C}}\otimes f^*(\mathcal{L}_{B \Gamma_2})}, D_{E_{\field{C}}|_{\partial W} \otimes g^*(\mathcal{L}_{B \Gamma_1})}), A )=0.$$
\normalsize
The result now follows from Equation \eqref{phiconeandindaps}.
\end{proof}

\begin{remark}
Lemma \ref{thelemmaformerlyknownasproposition3.12} should be compared with the results in \cite[Section 4]{DeeMappingCone}. When $\phi_*: K_*(C^*(\Gamma_1)) \rightarrow K_*(C^*(\Gamma_2))$ is injective, the proof of Lemma \ref{thelemmaformerlyknownasproposition3.12} implies that the isomorphism constructed here (i.e., $\Phi_{\rm cone}$) agrees with the isomorphism defined in \cite[Section 4]{DeeMappingCone}, i.e. the isomorphisms agree when both are defined. This follows from the general fact that $c(W,(E_{B_2},F_{B_1},\alpha),(D_E^W,D_F^{\partial W}),A)=0$ when we can pick the trivializing operator $A$ over $B_1$ (cf. Item (3) in the statement of Lemma \ref{thelemmaformerlyknownasproposition3.12}). We will not need this result in this paper and refrain from giving a detailed proof.
\end{remark}

\subsection{An application to PSC-metrics}
For a spin manifold with boundary that admits a metric of positive scalar curvature, more can be said about its relative assembly. A particularly nice feature of positive scalar curvature is that the Dirac operators appearing in the assembled cycles are invertible so $0$ provides a canonical choice of trivializing operator. In \cite[Theorem 2.18]{CWY} it was proven that its relative assembly by means of localization algebras vanishes. The two results are in several cases equivalent by Theorem \ref{comDiaGeoVsCWY} below. We provide a short proof of this fact in the geometric setting. We remain in $K$-homology although the same proof carries over to $KO$-homology. We now turn to Theorem \ref{thmpscone} in the Introduction.

\begin{theorem}
\label{vanishingpscresult}
Let $W$ be a connected spin-manifold with boundary and let $(W, W\times \field{C}, (\mathrm{id}_W, \mathrm{id}_{\partial W}))$ denote the cycle representing the fundamental class in $K_*^{\rm geo}(W,\partial W)$. If $W$ admits a metric of positive scalar curvature that is collared at the boundary, then 
$$\mu(W, W\times \field{C}, (\mathrm{id}_W, \mathrm{id}_{\partial W}))=0 \in K_{{\rm dim}(W)}(pt;\phi)$$
where $\phi: \pi_1(\partial W) \rightarrow \pi_1(W)$ is the group homomorphism induced from the inclusion $i:\partial W\hookrightarrow W$.
\end{theorem}

Here we are using the assembly mapping $\mu_{\rm geo}^\phi:K_*^{\rm geo}(W,\partial W)\to K_{*}^{\rm geo}(pt;\phi)$ defined from the composition of the relative assembly mapping with the functorially associated push forward $K_*^{\rm geo}(W,\partial W)\to K_*^{\rm geo}(B\phi)$. Theorem \ref{vanishingpscresult} will follow from the following two results.

\begin{prop}
\label{apswhenpsconbdry}
Let $W$ be a connected spin-manifold with boundary and let $(W, W\times \field{C}, (\mathrm{id}_W, \mathrm{id}_{\partial W}))$ denote the cycle representing the fundamental class in $K_*^{\rm geo}(W,\partial W)$. If $\partial W$ admits a metric $g_{\partial W}$ which has positive scalar curvature, then 
$$\mu(W, W\times \field{C}, (\mathrm{id}_W, \mathrm{id}_{\partial W}))=\Phi_{\rm cone}^{-1}(j_*({\rm ind}_{APS}(D^{ W}_{f^*(\mathcal{L}_{B\Gamma_2})},0))$$
where $D^{ W}_{f^*(\mathcal{L}_{B\Gamma_2})}$ is the twisting by the Mishchenko bundle of a spin-Dirac operator on $W$ constructed from a metric $g$ that takes the form $\mathrm{d} y^2+g_{\partial W}$ near the boundary, here $y$ denotes the normal coordinate.
\end{prop}

\begin{proof}
The existence of a metric of positive scalar curvature on the boundary implies that ${\rm ind}_{AS}(D^{\partial W}_{\mathcal{L}_{B\pi_1(\partial W)}})=0$. We let $g$ denote the extension of $g_{\partial W}$ to a metric on $W$. When constructed from $g$, $D^{\partial W}_{\mathcal{L}_{B\pi_1(\partial W)}}$ is invertible. By definition, the index of the Dirac operator $D^W_{\mathcal{L}_{B\pi_1(W)}}$  equipped with the Atiyah-Patodi-Singer boundary condition defined from the spectral section 
$$\chi_{[0,\infty)}(D^{\partial W}_{\mathcal{L}_{B\pi_1( W)}})=\alpha^*\left(\chi_{[0,\infty)}(D^{\partial W}_{\mathcal{L}_{B\pi_1(\partial W)}})\otimes_\phi 1_{C^*(\pi_1(W))}\right),$$
coincides with ${\rm ind}_{APS}(D^W_{\mathcal{L}_{B\pi_1(W)}},0)$. Using Lemma \ref{thelemmaformerlyknownasproposition3.12}, the proposition follows. 
\end{proof}

The next lemma is folklore. In lack of a precise reference, we provide a short proof of the result.

\begin{lemma}
\label{vanishingofaps}
If $W$ is a connected spin-manifold with boundary and $g$ is a metric of positive scalar curvature that is collared at the boundary, then 
$${\rm ind}_{APS}(D^W_{\mathcal{L}_{B\pi_1(W)}},0)=0\in K_*(C^*(\pi_1(W))),$$
where $D^W_{\mathcal{L}_{B\pi_1(W)}}$ denotes the spin-Dirac operator constructed from $g$ twisted by the flat connection on the Mishchenko bundle.
\end{lemma}

\begin{proof}
We write $W_\infty$ for $W$ glued together with the infinite cylinder $(1,\infty)\times \partial W$. Let $y$ denote the normal variable and write $g=\mathrm{d}y^2+g_{\partial W}$ near the boundary for some metric lifted from $\partial W$.  The metric $g$ can be extended to a metric $g_\infty$ on $W_\infty$ with positive scalar curvature by setting $g_\infty=\mathrm{d}y^2+g_{\partial W}$ on the cylinder $(1,\infty)\times \partial W$. We let $\mathcal{D}_{g,\infty}$ denote the $C^*(\pi_1(W))$-linear spin-Dirac operator on $W_\infty$ constructed from $g_\infty$ and twisted by the Mishchenko bundle. This is a self-adjoint regular operator because $g_\infty$ is complete. Since $g_\infty$ has scalar curvature with a positive uniform lower bound, $\mathcal{D}_{g,\infty}$ is invertible. Using the $b$-Mishchenko-Fomenko calculus (see \cite[Part III]{LPmono}), there is a well defined index class ${\rm ind}_{\rm MF}^b(\mathcal{D}_{g,\infty})\in K_*(C^*(\pi_1(W)))$ and ${\rm ind}_{\rm MF}^b(\mathcal{D}_{g,\infty})=0$ because $\mathcal{D}_{g,\infty}$ is invertible. By \cite[Proposition 2.4]{PSrhoInd},  ${\rm ind}_{APS}(D^W_{\mathcal{L}_{B\Gamma_2}},0)={\rm ind}_{\rm MF}^b(\mathcal{D}_{g,\infty})$ proving the lemma. 
\end{proof}

\subsection{The strong Novikov property}
\label{subsecononen}

The Novikov conjecture asserts that for an oriented manifold $M$, its higher signatures 
$$\mathrm{sign}_\nu(M,f):=\int_M f_M^*(\nu)\wedge L(M), \quad\nu\in H^*(B\pi_1(M)),$$ 
are oriented homotopy invariants of $M$. Here $L(M)$ denotes the $L$-class and $f_M$ the (up to homotopy) canonical mapping $M\to B\pi_1(M)$. A group $\Gamma$ is said to satisfy the Novikov conjecture if the higher signatures $\mathrm{sign}_\nu(M)$ are homotopy invariants of $M$ for any $\nu\in H^*(B\Gamma)$ and $f:M\to B\Gamma$. 

The group $\Gamma$ is said to satisfy the strong Novikov conjecture if the assembly map $\mu:K_*(B\Gamma)\to K_*(C^*(\Gamma))$ is injective. It follows from the homotopy invariance of higher indices (see \cite{hilsskand}) that for a given group the strong Novikov conjecture implies the Novikov conjecture. We shall study the relative setting: first we define relative higher signatures and we then show that rational injectivity of the relative assembly mapping implies homotopy invariance of relative higher signatures.

To define relative higher signatures, we use the relative cohomology groups of the mapping $B\phi:B\Gamma_1\to B\Gamma_2$. Let us briefly recall the construction of relative cohomology. As above, $h:Y\to X$ denotes a continuous mapping. The complex valued singular cochain complex of $X$ is denoted by $\mathcal{C}^*_{\mathrm{sing}}(X)$, and that of $Y$ by $\mathcal{C}^*_{\mathrm{sing}}(Y)$. The mapping $h$ pulls back cochains linearly $h^*:\mathcal{C}^*_{\mathrm{sing}}(X)\to \mathcal{C}^*_{\mathrm{sing}}(Y)$. The relative cohomology $H^*(h)$ of $h$ is the cohomology of the complex $\mathcal{C}^*_{\mathrm{sing}}(X)\oplus \mathcal{C}^{*+1}_{\mathrm{sing}}(Y)$ equipped with the coboundary operator
$$\partial^{\mathrm{rel}}:=\begin{pmatrix} \partial_X&0\\ -h^*& \partial_Y\end{pmatrix}.$$
In the special case that $h$ is an inclusion mapping, we write $H^*(X,Y)$ instead of $H^*_{\rm rel}(h)$. Specializing further to a manifold with boundary $W$, we note that if $h$ is the boundary inclusion, the de Rham map gives an isomorphism $H^*_{c,dR}(W^\circ)\to H^*(W,\partial W)$ from compactly supported de Rham cohomology to relative cohomology. In particular, for any closed form $\omega$ on $W$ we can define an integral $\int_W -\wedge \omega:H^*(W,\partial W)\to \C$. 

An important feature of relative cohomology is the a relative cohomology pairing with relative $K$-homology:
\[
\langle\cdot,\cdot\rangle_{\rm rel}:K_*^{\rm geo}(h)\times H^*_{\rm rel}(h)\to \C, \quad \langle(W,E,(f,g)),\lambda \rangle_{\rm rel}:=\int_W (f,g)^*\lambda\wedge \mathrm{ch}(E)\wedge \mathrm{Td}(W).
\]
By definition, the relative pairing is functorial in the sense that if $(f,g):[h:Y\to X]\to [h':Y'\to X']$ is a map of pairs, $x\in K_*^{\rm geo}(h)$ and $\lambda\in H^*_{\rm rel}(h')$ then 
\begin{equation}
\label{relcohompairfunc}
\langle(f,g)_*x,\lambda \rangle_{\rm rel}=\langle x,(f,g)^*\lambda \rangle_{\rm rel}.
\end{equation}

\begin{define}
Let $\nu\in H^*(B\phi)$ be a relative cohomology class. For an oriented manifold with boundary $W$ and map of pairs 
$$(f,g):[i:\partial W\to \overline{W}]\to [B\phi:B\Gamma_1\to B\Gamma_2],$$ 
the relative higher signature of $(W,(f,g))$ defined from $\nu$ is 
$$\mathrm{sign}_\nu(W,(f,g)):=\int_W (f,g)^*(\nu)\wedge L(W).$$
\end{define}

In the literature \cite{Lott,weinnov}, one finds analogues to the Novikov conjecture in the relative setting. Homotopy invariance of relative higher signatures of manifolds with boundary was discussed in \cite{leichpiahigher}. We will now consider a strong version of the relative Novikov property and prove a result about the homotopy invariance of relative higher signatures. 

\begin{define}
A group homomorphism $\phi:\Gamma_1\to \Gamma_2$ has the strong relative Novikov property if the free assembly mapping $\mu_{\rm geo}^\phi:K_*^{\rm geo}(B\phi)\to K_*^{\rm geo}(pt;\phi)$ is rationally injective. 
\end{define}

We remark that one could formulate the strong relative Novikov property using other $C^*$-completions of the group algebra. The functoriality of the maximal completion ensures that the formulation in the maximal completion is the weakest version of a strong relative Novikov property. Before discussing the implications of the strong relative Novikov property on homotopy invariance of relative higher signatures, we introduce some terminology. The reader should recall the notion of homotopy equivalence of pairs from Definition \ref{homotopiesofomor} (see page \pageref{homotopiesofomor}).

If $W$ and $W'$ are manifolds with boundary and $(f,g):[i:\partial W\to \overline{W}]\to [i':\partial W'\to \overline{W}']$ is a morphism, we say that $(f,g)$ is of product type near the boundary if $f|_{\partial W\times (0,1]}=g\times\mathrm{id}_{(0,1]}$. Here we identify $\partial W\times (0,1]$ and $\partial W'\times (0,1]$ with collar neighborhoods of the boundaries in $\overline{W}$ and $\overline{W}'$, respectively. If $(f,g):[i:\partial W\to \overline{W}]\to [i':\partial W'\to \overline{W}']$ is an arbitrary morphism, we can extend to a morphism $(f_2,g_2):[i:\partial W\to \overline{W}_2]\to [i':\partial W'\to \overline{W}'_2]$ of product type near the boundary. After conjugating by homeomorphisms $\overline{W}\cong \overline{W}_2$ and $\overline{W}'\cong \overline{W}'_2$, we obtain a morphism of product type $(\tilde{f},\tilde{g}):[i:\partial W\to \overline{W}]\to [i':\partial W'\to \overline{W}']$. The following proposition follows from applying the construction $(f,g)\mapsto (\tilde{f},\tilde{g})$ to a homotopy equivalence.

\begin{prop}
\label{produhom}
Let $W$ and $W'$ be manifolds with boundary. If $[i:\partial W\to \overline{W}]$ and $[i':\partial W'\to \overline{W}']$ are homotopy equivalent, then there is a homotopy equivalence of  $[i:\partial W\to \overline{W}]$ and $[i':\partial W'\to \overline{W}']$  through morphisms of product type near the boundary.
\end{prop}

\begin{theorem}
\label{homotopyandnovi}
Suppose $\phi:\Gamma_1\to \Gamma_2$ is a group homomorphism with the strong relative Novikov property. Then for any manifold with boundary $W$ and map of pairs 
$$(f,g):[i:\partial W\to \overline{W}]\to [B\phi:B\Gamma_1\to B\Gamma_2],$$ 
the relative higher signatures
\begin{equation}
\label{homotopequ}
\mathrm{sign}_\nu(W, (f,g)):=\int_W (f,g)^*(\nu)\wedge L(W), \quad\nu\in H^*(B\phi),
\end{equation}
are orientation preserving homotopy invariants of $[i:\partial W\to \overline{W}]$. 
\end{theorem}

Prior to the proof, we discuss what is meant by the relative higher signatures being homotopy invariants of $[i:\partial W\to \overline{W}]$. Let $(u,u_\partial ):[i:\partial W\to \overline{W}]\to [i':\partial W'\to \overline{W}']$ be an orientation preserving homotopy equivalence which we by Proposition \ref{produhom} can assume to be of product type. The morphism $(f,g):[i:\partial W\to \overline{W}]\to [B\phi:B\Gamma_1\to B\Gamma_2]$ is up to homotopy determined by a universal morphism $[i:\partial W\to \overline{W}]\to [Bi:B\pi_1(\partial W)\to B\pi_1(\overline{W})]$ and the group homomorphisms $q_2:=Bf_*:\pi_1(\overline{W})\to \pi_1(B\Gamma_2)= \Gamma_2$ and $q_1:=Bg_*:\pi_1(\partial W)\to \pi_1(B\Gamma_1)=\Gamma_1$. Let $(f',g'):[i:\partial W'\to \overline{W}']\to [B\phi:B\Gamma_1\to B\Gamma_2]$ denote the composition of $(f,g)$ with $(u,u_\partial )$. The statement of Theorem \ref{homotopyandnovi} is that 
$$\int_W (f,g)^*(\nu)\wedge L(W)=\int_{W'} (f',g')^*(\nu)\wedge L(W'), \quad\forall \nu\in H^*(B\phi).$$

\begin{proof}[Proof of Theorem \ref{homotopyandnovi}]
We let ${\rm Sign}_W\in K_*^{\rm geo}(W,\partial W)$ and ${\rm Sign}_{W'}\in K_*^{\rm geo}(W',\partial W')$ denote the $K$-homology classes defined from the signature operators. The reader is referred to \cite{bgv,leichpiahigher} for details on the signature operator. To describe the classes $\mathrm{Sign}_W$ and $\mathrm{Sign}_{W'}$ in geometric $K$-homology, we use its oriented model as discussed in Remark \ref{firstorientedremark} and \ref{secondorientedremark}. The Hodge star will be denoted by $\star$. We follow the convention in \cite[Chapter 3.6]{bgv} and normalize $\star$ so that $\star^2=1$. In the oriented model for geometric $K$-homology, see \cite{guentnerkhom,Kescont}, the class ${\rm Sign}_W$ is represented by the relative oriented cycle $(W,\wedge^*_0 W,(\mathrm{id}_W,\mathrm{id}_{\partial W}))$ where $\wedge^*_0 W$ is the Clifford bundle defined from the exterior algebra $\wedge^*W$ and the Hodge star, i.e. $\wedge^*_0 W=\wedge^* W$ graded by the Hodge star in the even-dimensional case and $\wedge^*_0 W=\ker(\star-1:\wedge^* W\to \wedge ^*W)$ in the odd-dimensional case. An analogous expression describes ${\rm Sign}_{W'}$. For any $\nu\in H^*(B\phi)$, we have that $\mathrm{sign}_\nu(W,(f,g))=\langle {\rm Sign}_W,(f,g)^*\nu\rangle_{\rm rel}$ and similarly for ${\rm sign}_\nu(W',(f',g'))$. The theorem therefore follows from Equation \eqref{relcohompairfunc} once $(f,g)_*\mathrm{Sign}_W=(f',g')_*\mathrm{Sign}_{W'}$ in $K_*^{\rm geo}(B\phi)\otimes \mathbb{Q}$. By assumption, $\mu_{\rm geo}^\phi$ is rationally injective, so it suffices to prove that $\mu_{\rm geo}^\phi\left[(f,g)_*\mathrm{Sign}_W\right]=\mu_{\rm geo}^\phi\left[(f',g')_*\mathrm{Sign}_{W'}\right]$ in $K_*^{\rm geo}(pt;\phi)$.

Consider the class $x=\mu_{\rm geo}^\phi\left[(f,g)_*\mathrm{Sign}_W\right]-\mu_{\rm geo}^\phi\left[(f',g')_*\mathrm{Sign}_{W'}\right]\in K_*^{\rm geo}(pt;\phi)$. The class $x$ is represented by the following cycle in the oriented model
\begin{align*}
(W, (\wedge^*_0 W\otimes f^*(\mathcal{L}_{B\Gamma_2}) ,&\wedge^*_0 W|_{\partial W}\otimes g^*(\mathcal{L}_{B\Gamma_1}) ,\alpha))-\\
&-(W', (\wedge^*_0 W'\otimes f'^*(\mathcal{L}_{B\Gamma_2}) ,\wedge^*_0 W'|_{\partial W'} \otimes g'^*(\mathcal{L}_{B\Gamma_1}) ,\alpha')).
\end{align*}
By homotopy invariance of the higher index of the signature operator, see for instance \cite[Theorem 6.3]{Kas}, we have that 
$$\mathrm{ind}_{AS}(\partial W,\wedge^*_0 W|_{\partial W}\otimes g^*(\mathcal{L}_{B\Gamma_1}))=\mathrm{ind}_{AS}(\partial W',\wedge^*_0 W'|_{\partial W'}\otimes g'^*(\mathcal{L}_{B\Gamma_1}))\in K_*(C^*(\Gamma_1)).$$
By Proposition \ref{thelemmaformerlyknownasproposition3.12}, we can realize $x$ as an APS-index of the form 
$$x=\Phi_{\rm cone}^{-1}(j_*({\rm ind}_{APS}(D^{ W\cup -W'}_{\rm sign} ,(\alpha\cup \alpha')\phi_*A))),$$
where $D^{ W\cup -W'}_{\rm sign}$ denotes the $C^*(\Gamma_2)$-linear signature operator on the $C^*(\Gamma_2)$-Clifford bundle 
$$\wedge^*_0 W\otimes f^*(\mathcal{L}_{B\Gamma_2})\cup (-\wedge^*_0 W'\otimes f'^*(\mathcal{L}_{B\Gamma_2}))\to W\cup -W',$$
and $A$ is a trivializing operator for the signature operator on 
$$\wedge^*_0 W|_{\partial W}\otimes g^*(\mathcal{L}_{B\Gamma_1})\cup (-\wedge^*_0 W'|_{\partial W'}\otimes g'^*(\mathcal{L}_{B\Gamma_1}))\to \partial W\cup -\partial W'.$$
The fact that $(u,u_\partial)$ is of product type near the boundary allows us to use \cite[Theorem 8.4]{WahlSurSet} which implies that ${\rm ind}_{APS}(D^{ W\cup -W'}_{\rm sign} ,(\alpha\cup \alpha')\phi_*A)=0$. We conclude that $x=0$.
\end{proof}

\section{Comparing assembly maps} \label{comAssMap}
In this section we prove the following theorem:
\begin{theorem} \label{comDiaGeoVsCWY}
Let $\phi: \Gamma_1 \rightarrow \Gamma_2$ be a group homomorphism and $B\phi: B\Gamma_1 \rightarrow B\Gamma_2$ the map $\phi$ induces at the classifying space level. In an abuse of notation, let $\phi$ also denote the $*$-homomorphism induced at the maximal $C^*$-algebra level and $B\phi$ the $*$-homomorphism $C^*_L(B\Gamma_1)\to C^*_L(B\Gamma_2)$ induced at the maximal localization algebra level. We let $\psi:C^*(E\Gamma_1)^{\Gamma_1}\to C^*(E\Gamma_2)^{\Gamma_2}$ denote the $*$-homomorphism constructed from $B\phi$ and $C_\psi$ its mapping cone. If one of the following conditions hold:
\begin{itemize}
\item The mapping $\phi_*:K_*(C^*(\Gamma_1))\to K_*(C^*(\Gamma_2))$ is (rationally) surjective.
\item The mapping $(B\phi)_*:K_*(B\Gamma_1)\to K_*(B\Gamma_2)$ is (rationally) injective.
\item For a finite $CW$-pair $(X,Y)$, $\Gamma_1=\pi_1(Y)$, $\Gamma_2=\pi_1(X)$ and $\phi$ being induced by the inclusion, the natural mapping $K_*(C_\phi)\to K_*(SC^*(\pi_1(X/Y)))\oplus K_*(C^*(\Gamma_1))$ is (rationally) injective.
\end{itemize}
then the following diagram commutes (rationally)
\begin{equation}
\label{commutingdiagrem}
\begin{CD}
K^{\textnormal{geo}}_*(B \phi) @>\mu_{\textnormal{geo}} >> K^{\textnormal{geo}}_*(pt;\phi)  \\
@V\mathrm{ind}_L^{\rm rel} VV @VV\Phi_{{\rm cone}}V   \\
K_{*+1}(C_{B \phi}) @>\mu_{CWY} >> K_{*+1}(C_{\psi}) \\
\end{CD}
\end{equation}
where the horizontal maps are the assembly maps (defined respectively in Sections \ref{CWYmap} and \ref{GeoMap}) and the vertical maps are the isomorphisms defined respectively in Theorem \ref{relaassonkhom} and Equation \eqref{definofphicone} in Lemma \ref{cyctoclasses}. 
\end{theorem}

\begin{remark}
We note that $(B\phi)_*:K_*(B\Gamma_1)\to K_*(B\Gamma_2)$ is rationally injective whenever the strong Novikov conjecture holds and  $\phi_*:K_*(C^*(\Gamma_1))\to K_*(C^*(\Gamma_2))$ is rationally injective. Moreover (the last condition appearing in Theorem \ref{comDiaGeoVsCWY}) the injectivity of $K_*(C_\phi)\to K_*(SC^*(\pi_1(X/Y)))\oplus K_*(C^*(\pi_1(Y)))$, is equivalent to the exactness of 
$$K_{*}(C^*(\pi_1(Y)))\to K_*(C^*(\pi_1(X)))\to K_*(C^*(\pi_1(X/Y))).$$

\end{remark}

We now turn to the proof of Theorem \ref{comDiaGeoVsCWY}. The idea is to use the absolute case to prove the commutativity of the diagram \eqref{commutingdiagrem}. It remains to verify that the assumptions in Theorem \ref{comDiaGeoVsCWY} imply the commutativity of \eqref{commutingdiagrem}.

\begin{lemma}
If the mapping $\phi_*:K_*(C^*(\Gamma_1))\to K_*(C^*(\Gamma_2))$ is (rationally) surjective, then the diagram \eqref{commutingdiagrem} (rationally) commutes.
\end{lemma}

\begin{proof}
Consider an $x\in K^{\textnormal{geo}}_*(B \phi)$ and set $y:=(\Phi_{{\rm cone}}\circ \mu_{\textnormal{geo}}-\mu_{CWY}\circ \mathrm{ind}_L^{\rm rel})x\in K_{*+1}(C_{\psi})$. We need to show that $y$ vanishes (rationally). Since the diagram \eqref{commutingdiagrem} commutes in the absolute case, $\delta(y)=0\in K_{*-1}(C^*(\Gamma_1))$. However, if $\phi_*:K_*(C^*(\Gamma_1))\to K_*(C^*(\Gamma_2))$ is (rationally) surjective then $\delta$ is (rationally) injective and $y$ vanishes (rationally). 
\end{proof}

\begin{lemma}
If the mapping $(B\phi)_*:K_*(B\Gamma_1)\to K_*(B\Gamma_2)$ is (rationally) injective, then the diagram \eqref{commutingdiagrem} (rationally) commutes.
\end{lemma}

\begin{proof}
If $(B\phi)_*:K_*(B\Gamma_1)\to K_*(B\Gamma_2)$ is (rationally) injective, then $r:K_*(B\Gamma_2)\to K_*(B\phi)$ is (rationally) surjective. In particular, if $x\in K_*(B\phi)$ there is a $y\in K_*(B\Gamma_2)$ such that $x=r(y)$ (rationally). By naturality, 
$$(\Phi_{{\rm cone}}\circ \mu_{\textnormal{geo}}-\mu_{CWY}\circ \mathrm{ind}_L^{\rm rel})x=r((\Phi_{{C^*(\Gamma_2)}}\circ \mu_{\textnormal{geo}}-\mu_{CWY}\circ \mathrm{ind}_L^{\rm rel})y)=r(0)=0.$$
\end{proof}

\begin{lemma}
Let $(X,Y)$ be a finite $CW$-pair and set $\Gamma_1:=\pi_1(Y)$, $\Gamma_2:=\pi_1(X)$. Let $\phi$ denote the mapping induced by the inclusion. If the natural mapping $K_*(C_\phi)\to K_*(SC^*(\pi_1(X/Y)))\oplus K_*(C^*(\Gamma_1))$ is (rationally) injective, then the diagram \eqref{commutingdiagrem} (rationally) commutes.
\end{lemma}

\begin{proof}
Define $\tau:=\Phi_{{\rm cone}}\circ \mu_{\textnormal{geo}}-\mu_{CWY}\circ \mathrm{ind}_L^{\rm rel}:K_*^{\textrm{geo}}(X,Y)\to K_*(C_\phi)$, where $\phi:=\pi_1(i)$ and $i:Y\to X$ denotes the inclusion. The mapping $\tau$ is a natural transformation that vanishes in the absolute case $Y=\emptyset$. We arrive at a commuting diagram
$$\begin{CD}
K^{\textnormal{geo}}_*(X,Y) @>>> K^{\textnormal{geo}}_{*+1}(X/Y)\oplus K^{\textnormal{geo}}_*(Y)  \\
@V\tau VV @VV0V   \\
K_{*}(C_{\phi}) @>>> K_*(SC^*(\pi_1(X/Y)))\oplus K_*(C^*(\Gamma_1)) \\
\end{CD}$$
Since the bottom horizontal arrow is injective, and the right vertical arrow is the zero mapping also the left vertical arrow is the zero mapping.
\end{proof}

\begin{remark}
Note that $\tau$ from the previous proof maps 
$$\tau:K_*^{\textrm{geo}}(X,Y)\to \ker(\delta:K_*(C_\phi)\to K_{*-1}(C^*(\Gamma_1))).$$ 
If $\tau$ lifts to a natural transformation $\hat{\tau}:K_*^{\textrm{geo}}(X,Y)\to K_*(C^*(\pi_1(X)))$ of functors on $CW$-pairs, then $\hat{\tau}=0$ by naturality and the diagram \eqref{commutingdiagrem} commutes. 
\end{remark}

\noindent
{\bf Acknowledgments} \\
The authors thank Stanley Chang, Shmuel Weinberger, and Guoliang Yu for providing them with a preliminary version of their paper \cite{CWY}. We also thank Thomas Schick for discussions. We are grateful to the referee of an earlier version of this paper; their comments have improved the paper substantially. The first listed author thanks the Fields Institute ``Workshop on Stratified Spaces: Perspectives from Analysis, Geometry and Topology" for funding. The second listed author thanks Centre for Symmetry and Deformation (Copenhagen) for travel support, the Knut and Alice Wallenberg foundation for support and the Swedish Research Council Grant 2015-00137 and Marie Sklodowska Curie Actions, Cofund, Project INCA 600398. The authors thank the Leibniz Universit\"at Hannover, the Graduiertenkolleg 1463 (\emph{Analysis, Geometry and String Theory}), Universit\'e Blaise Pascal Clermont-Ferrand, and the Focus Programme on $C^*$-algebras held at the University of M${\rm \ddot{u}}$nster in 2015 for facilitating this collaboration.

\vspace{0.25cm}
\noindent
Email address: robin.deeley@gmail.com \vspace{0.25cm} \\
{ \footnotesize Department of Mathematics, University of Colorado Boulder Campus Box 395, Boulder, CO 80309-0395, USA }
\vspace{.5cm}\\
Email address: goffeng@chalmers.se \vspace{0.25cm} \\
{ \footnotesize Department of Mathematical Sciences, Chalmers Tv\"argata 3, 412 96 G\"oteborg, Sweden }

\begin{thebibliography}{99}
\bibitem{APSI} M. F. Atiyah, V. K. Patodi, and I. M. Singer, {\it Spectral asymmetry and Riemannian geometry. I}, Math. Proc. Cambridge Philos. Soc. 77 (1975), 43--69.
\bibitem{BCW} P. Baum, A. Carey, B. L. Wang, \emph{$K$-cycles for twisted $K$-homology}, J. $K$-theory, 12 (2013) no. 1, 69-98.
\bibitem{bch} P. Baum, A. Connes, and N. Higson, {\it Classifying space for proper actions and $K$-theory of group $C^\ast$-algebras}, $C^\ast$-algebras: 1943--1993 (San Antonio, TX, 1993), 240--291, Contemp. Math., 167, Amer. Math. Soc., Providence, RI, 1994.
\bibitem{BD}P. Baum and R. Douglas. {\it $K$-homology and index theory}. Operator Algebras and Applications (R. Kadison editor), volume 38 of Proceedings of Symposia in Pure Math., 117-173, Providence RI, 1982. AMS.
\bibitem{BDrelCstar} P. Baum and R. Douglas. {\it Relative $K$-homology and $C^*$-algebras}, $K$-theory, 5:1-46, 1991.
\bibitem{BDT} P. Baum, R. Douglas, and M. Taylor. {\it Cycles and relative cycles in analytic $K$-homology}. J. Diff. Geo., 30: 761-804 1989.
\bibitem{BE} P. Baum and E. van Erp. {\it $K$-homology and index theory on contact manifolds}, Acta Math. 213 (2014), no. 1, p. 1--48.
\bibitem{BHS}P. Baum, N. Higson, and T. Schick. {\it On the equivalence of geometric and analytic $K$-homology}. Pure Appl. Math. Q. 3: 1-24, 2007 
\bibitem{bgv} N. Berline, E. Getzler, and M. Vergne, {\it Heat kernels and Dirac operators}, Grundlehren der Mathematischen Wissenschaften [Fundamental Principles of Mathematical Sciences], 298. Springer-Verlag, Berlin, 1992.
\bibitem{blackbook} B. Blackadar, \emph{$K$-theory for operator algebras}, Second edition. Mathematical Sciences Research Institute Publications, 5. Cambridge University Press, Cambridge, 1998.
\bibitem{CWY}S. Chang, S. Weinberger, and G. Yu. {\it Positive scalar curvature and a new index theory for noncompact manifolds}, arXiv:1506.03859.
\bibitem{chernoffess} P. R. Chernoff, {\it Essential self-adjointness of powers of generators of hyperbolic equations}, J. Functional Analysis 12 (1973), p. 401--414.
\bibitem{connesbook} A. Connes, {\it Noncommutative geometry}, Academic Press, Inc., San Diego, CA, 1994. {\rm xiv}+661 pp. 
\bibitem{conneshigson} A. Connes, and N. Higson, {\it  D\'eformations, morphismes asymptotiques et $K$-th\'eorie bivariante}, C. R. Acad. Sci. Paris S\'er. I Math. 311 (1990), no. 2, 101--106.
\bibitem{DeeZkz}R. J. Deeley. {\it Geometric $K$-homology with coefficients I: $\zkz$-cycles and Bockstein sequence}. J. $K$-theory, 9 (2012) no. 3, p. 537--564.
\bibitem{DeeRZ}R. J. Deeley. {\it $\mathbb{R}/\Z$-valued index theory via geometric $K$-homology}, M\"unster Journal of Math., {\bf 7} 2014 697--729.
\bibitem{DeeMappingCone}R. J. Deeley. {\it Analytic and topological index maps with values in the $K$-theory of mapping cones}, Journal of Noncommutative geometry {\bf 10} 2016 2:681--708.
\bibitem{DGMbor} R. J. Deeley, M. Goffeng and B. Mesland, \emph{The bordism group of unbounded $KK$-cycles},  J. Topol. Anal. {\bf 10} (2018) no 2 355--400.
\bibitem{DG} R. J. Deeley and M. Goffeng, \emph{Realizing the analytic surgery group of Higson and Roe geometrically, Part I: The geometric model}, J. Homotopy Relat. Struct. {\bf 12} (2017), no. 1, 109--142.
\bibitem{DGII} R. Deeley and M. Goffeng. {\it Realizing the analytic surgery group of Higson and Roe geometrically, Part II: Relative $\eta$-invariants}, Math. Ann. 366 (3) 2016 1319--1363.
\bibitem{DGIII} R. J. Deeley and M. Goffeng, \emph{Realizing the analytic surgery group of Higson and Roe geometrically, Part III: Higher invariants}, Math. Ann. 366 (3) 2016 1513--1559.
\bibitem{DGrelII} R. J. Deeley and M. Goffeng, \emph{Relative geometric assembly and mapping cones, Part II: Chern characters}, arXiv:1705.08467.
\bibitem{gongwangyu} G. Gong, Q. Wang and G. Yu, {\it Geometrization of the strong Novikov conjecture for residually finite groups}, J. Reine Angew. Math. 621 (2008), p. 159--189.
\bibitem{guentnerkhom} E. Guentner, {\it $K$-homology and the index theorem}, Index theory and operator algebras (Boulder, CO, 1991), p. 47--66. Contemp. Math., 148, Amer. Math. Soc., Providence, RI, 1993. 
\bibitem{higroeyu} N. Higson, J. Roe and G. Yu, {\it A coarse Mayer-Vietoris principle}, Math. Proc. Cambridge Philos. Soc. 114 (1993), no. 1, p. 85--97.
\bibitem{hilsskand} M. Hilsum and G. Skandalis, {\it Invariance par homotopie de la signature \`a coefficients dans un fibr\'e presque plat}, J. reine angew. Math. 423 (1992), p. 73--99
\bibitem{Jak}M. Jakob. {\it A bordism-type construction of homology}. Manuser. Math. 96: 67-80 1998.
\bibitem{Kas} G. Kasparov. {\it Equivariant $KK$-theory and the Novikov conjecture}. Invent. Math. 91, 147-201, 1998.
\bibitem{Kescont} N. Keswani. {\it Geometric $K$-homology and controlled paths}, New York J. Math. 5 (1999), p. 53--81.
\bibitem{land} M. Land, \emph{The Analytical Assembly Map and Index Theory},  J. Noncommut. Geom. 9 (2015), no. 2, p. 603--619.
\bibitem{LPmono} E. Leichtnam and P. Piazza, {\it The $b$-pseudodifferential calculus on Galois coverings and a higher Atiyah-Patodi-Singer index theorem}, M\'em. Soc. Math. Fr. (N.S.) No. 68 (1997).
\bibitem{LPGAFA} E. Leichtnam and P. Piazza. {\it Spectral sections and higher Atiyah-Patodi-Singer index theory on Galois coverings}, Geom. Funct. Anal. 8 (1998), no. 1, p. 17--58. 
\bibitem{LP}E. Leichtnam and P. Piazza. {\it Dirac index classes and the noncommutative spectral flow}. J. Funct. Anal. 200 no. 2, 348-400, 2003.
\bibitem{LPpsc} E. Leichtnam, P. Piazza. {\it On higher eta invariants and metrics of positive scalar curvature}, $K$-Theory, {\bf 24} 2001.
\bibitem{leichpiahigher} E. Leichtnam, and P. Piazza, {\it Elliptic operators and higher signatures}, Ann. Inst. Fourier (Grenoble) 54 (2004), no. 5, p. 1197--1277.
\bibitem{Lott} J. Lott. {\it Higher eta invariants}, $K$-Theory {\bf 6} 191-233, 1992.
\bibitem{lottsuper} J. Lott. {\it Superconnections and higher index theory}, GAFA. 2 (1992), no. 4, p. 421--454. 
\bibitem{PS} P. Piazza and T. Schick. {\it Bordism, rho-invariants and the Baum-Connes conjecture}. J. Noncommut. Geom. 1 no. 1, p. 27--111, 2007.
\bibitem{PSrhoInd} P. Piazza and T. Schick. {\it Rho-classes, index theory and Stolz' positive scalar curvature sequence}, J. Topology (2014) 7 (4) 965--1004.
\bibitem{qr10} Y. Qiao and J. Roe, {\it On the localization algebra of Guoliang Yu}, Forum Math. 22 (2010), p. 657--665.
\bibitem{Rav} J. Raven. {\it An equivariant bivariant Chern character}, PhD Thesis, Pennsylvania State University, 2004. (available online at the Pennsylvania State Digital Library).
\bibitem{Roe88} J. Roe, {\it Partitioning noncompact manifolds and the dual Toeplitz problem}, Operator algebras and applications, Vol. 1, p. 187--228, London Math. Soc. Lecture Note Ser., 135, Cambridge Univ. Press, Cambridge, 1988.
\bibitem{Roe96} J. Roe. {\it Index theory, coarse geometry, and topology of manifolds} CBMS Regional Conference Series in Mathematics 90 1996.
\bibitem{roeasympt} J. Roe, {\it Elliptic operators, topology and asymptotic methods}, Second edition. Pitman Research Notes in Mathematics Series, 395. Longman, Harlow, 1998. 
\bibitem{SchickSeyedhosseini} T. Schick, and M. Seyedhosseini, {\it On an Index Theorem of Chang, Weinberger and Yu}, preprint 2018.
\bibitem{WahlSurSet} C. Wahl. {\it Higher $\rho$-invariants and the surgery structure set},  J. Topol. 6 (2013), no. 1, p. 154--192.
\bibitem{Wal} M. Walter. {\it Equivariant geometric $K$-homology with coefficients}. Diplomarbeit University of Goettingen. 2010.
\bibitem{weinnov} S. Weinberger, {\it Aspects of the Novikov conjecture}, Geometric and topological invariants of elliptic operators (Brunswick, ME, 1988), p. 281--297, Contemp. Math., 105, Amer. Math. Soc., Providence, RI, 1990.
\bibitem{willettyu12} R. Willett, G. Yu, {\it Higher index theory for certain expanders and Gromov monster groups, I.} Adv. Math. 229 (2012), no. 3, p. 1380--1416. 
\bibitem{xieyu} Z. Xie and G. Yu. {\it Positive scalar curvature, higher $\rho$-invariants and localization algebras}, Adv. Math. 262 (2014), p. 823--866. 
\bibitem{Yu} G. Yu, {\it Localization algebras and the coarse Baum-Connes conjecture}, $K$-theory {\bf 11} no. 4, 307-318, 1997.
\bibitem{zeid14} R. Zeidler, {\it Positive scalar curvature and product formulas for secondary index invariants},  J. Topology (2016) 9 (3): 687--724. 
\end{thebibliography}
\end{document}